\documentclass{amsart}

\usepackage[latin1]{inputenc}
\usepackage{enumerate}
\usepackage[T1]{fontenc}
\usepackage[english]{babel}
\usepackage{amsmath}
\usepackage{verbatim}
\usepackage{amsfonts}
\usepackage[all]{xy}
\usepackage{amssymb}
\usepackage{tikz}
\usetikzlibrary{arrows}
\usepackage{hyperref}
\usepackage{mathrsfs}
\usepackage{color}

\numberwithin{equation}{section}

\newcommand{\incl}[1][r]{\ar@<-0.2pc>@{^(-}[#1] \ar@<+0.2pc>@{-}[#1]}

\newcommand{\sgeq}{\mathbin{\geq\hspace{-1.06em}_{_(\:\,_)}}}
\newcommand{\sleq}{\mathbin{\leq\hspace{-1.06em}_{_(\:\,_)}}}

\newcommand{\Ker}{\operatorname{Ker}}
\newcommand{\Irr}{\operatorname{Irr}}
\newcommand{\rk}{\operatorname{rk}}

\newcommand{\Hilb}{\operatorname{Hilb}}
\newcommand{\Spec}{\operatorname{Spec}}

\newcommand{\GQuot}{\mathrm{Quot}^G(\mathcal{H},h)}
\newcommand{\GL}{\mathrm{GL}}
\newcommand{\SL}{\mathrm{SL}}

\newcommand{\Coh}{\mathrm{Coh}}

\newcommand{\supp}{\operatorname{supp}}

\newcommand{\FF}{\mathcal{F}}

\newcommand{\OO}{\mathcal{O}}
\newcommand{\HH}{\mathcal{H}}

\newcommand{\KK}{\mathcal{K}}
\newcommand{\JJ}{\mathcal{J}}
\newcommand{\GG}{\mathcal{G}}
\newcommand{\NN}{\mathbb{N}}
\newcommand{\Ll}{\mathscr{L}}

\newcommand{\Gm}{\mathbb{G}_m}

\newcommand{\QQ}{\mathbb{Q}}

\newcommand{\RR}{\mathbb{R}}
\newcommand{\PP}{\mathbb{P}}
\newcommand{\ZZ}{\mathbb{Z}}
\newcommand{\CC}{\mathbb{C}}

\newcommand{\mut}{\mu_{\theta}}
\newcommand{\odd}{\operatorname{odd}}
\newcommand{\even}{\operatorname{even}}

\theoremstyle{plain}
\newtheorem{theorem}{Theorem}[section]
\newtheorem{lemma}{Lemma}[section]
\newtheorem{proposition}{Proposition}[section]
\newtheorem{corollary}{Corollary}[section]
\newtheorem{example}{Example}[section]
\newtheorem{hypothesis}{Hypothesis}[section]
\newtheorem*{theorem*}{Theorem}
\newtheorem{question}{Question}[section]
\newtheorem{warning}{Warning}[section]
\newtheorem*{proposition**}{Proposition A}
\newtheorem*{theorem**}{Theorem B}

\theoremstyle{definition}
\newtheorem{definition}{Definition}[section]

\theoremstyle{remark}
\newtheorem{remark}{Remark}[section]

\address{Ronan Terpereau \newline 
Former institution: Fachbereich Physik, Mathematik und Informatik, Johannes Gutenberg -- Universit\"at Mainz , 55099 Mainz, Germany \newline
 Present institution: Institut de Math\'ematiques de Bourgogne - UMR 5584 du CNRS, Universit\'e Bourgogne Franche-Comt\'{e}, F-21000 Dijon, France}
\email{\noindent ronan.terpereau@u-bourgogne.fr}

\address{Alfonso Zamora \newline
Former institution: Center for Mathematical Analysis, Geometry and Dynamical Systems, Instituto Superior T\'ecnico, Universidade de Lisboa, Av. Rovisco Pais, 1049-001 Lisboa, Portugal \newline
Present institution: Departamento Interfacultativo de Matem\'atica Aplicada y Estad\'istica, Facultad de Ciencias Econ\'omicas y Empresariales, Universidad CEU San Pablo, Juli\'an Romea 23, Madrid, Spain} 
\email{\noindent alfonso.zamorasaiz@ceu.es}



\begin{document}

\title[Stability conditions and related filtrations for constellations]{Stability conditions and related filtrations for $(G,h)$-constellations}

\author{Ronan Terpereau}
\author{Alfonso Zamora}

\begin{abstract}
Given an infinite reductive algebraic group $G$, we consider $G$-equivariant coherent sheaves with 
prescribed multiplicities, called $(G,h)$-constel\-lations, for which two stability notions arise. 
The first one is analogous to the $\theta$-stability defined for quiver representations by King \cite{Kin:1994} 
and for $G$-constellations by Craw and Ishii \cite{CI04}, but depending on infinitely many parameters. The second one 
comes from Geome\-tric Invariant Theory in the construction of a moduli space for $(G,h)$-constellations, and depends on some finite subset $D$ 
of the isomorphy classes of irreducible representations of $G$.  
We show that these two stability notions do not coincide, answering negatively a question raised in \cite{BT15}. 
Also, we construct Harder-Narasimhan filtrations for $(G,h)$-constellations with respect to both stability notions (namely, 
the $\mut$-HN and $\mu_D$-HN filtrations). 
Even though these filtrations do not coincide in general, we prove that they are strongly related: 
the $\mut$-HN filtration is a subfiltration of the $\mu_D$-HN filtration, and the polygons of the $\mu_D$-HN filtrations
converge to the polygon of the $\mut$-HN filtration when $D$ grows.
\end{abstract}

\maketitle

\tableofcontents

\section*{Introduction}
In moduli problems, we usually consider objects on which we impose a certain stability condition to 
be able to construct a moduli space, i.e., a space parametrizing stable or semistable objects. 
When using Geometric Invariant Theory (in the following GIT to abbreviate) to construct such a moduli space,
another notion of stability shows up, the so-called GIT-stability, and one has to work out the relation between those two stability notions
in order to apply GIT to construct the moduli space of (semi)stables objects we are interested in. In this article, we consider a new moduli problem 
treated in \cite{BT15}, where the objects are certain coherent sheaves for which the stability condition depends 
on infinitely many parameters and do not exactly match with the GIT-stability condition. We investigate in detail
the two stability conditions and notice remarkable phenomena, in particular at the 
level of the corresponding Harder-Narasimhan filtrations. Moreover, we answer several questions which remained unresolved in \cite{BT15}. 
 
Let $G$ be a complex reductive algebraic group, and let $X$ be an affine $G$-scheme of finite type. When $G$ is finite, Craw and Ishii \cite{CI04} 
generalized the notion of $G$-cluster on $X$. They defined a \emph{$G$-constellation} on $X$ as a $G$-equivariant coherent $\OO_X$-module $\FF$ 
with global sections $H^0(\FF)$ isomorphic to the regular representation of $G$ as a $G$-module. Then they defined a stability condition on 
$G$-constellations, namely the \emph{$\theta$-stability}, and they constructed the moduli space of $\theta$-(semi)stable $G$-constellations on $X$ 
by following ideas of King \cite{Kin:1994}. The key ingredient in this construction is the reformulation of the $\theta$-stability condition into 
a GIT-stability condition. 
Finally, they proved that minimal resolutions of singularities of certain quotients $X/G$ can be obtained as the moduli space of $\theta$-stable 
$G$-constellations on $X$ for some $\theta$.

Let us now assume that $G$ is infinite. In \cite{BT15}, Becker and the first-named author defined similar concepts and constructed the moduli space 
of $\theta$-stable $(G,h)$-constellations on $X$, where $h: \Irr G \to \NN_{\geq 0}$ is a function that assigns a non-negative integer to each 
irreducible $G$-module and replaces the regular representation in this new setting; see \S \ref{def_and_HN} for details. 
This moduli space, say $M_{\theta}(X)$, is a generalization of the invariant Hilbert scheme of Alexeev and Brion \cite{AB05}. 
Also, by analogy with the case where $G$ is a finite abelian subgroup of $\SL_3(\CC)$ (see \cite[Theorem 1.1]{CI04}), it is expected that 
any crepant resolution of singularities of the categorical quotient $X/\!/G$ is isomorphic to $M_{\theta}(X)$ for some $\theta$ (when such a resolution exists).

The methods used in \cite{CI04} to construct $M_{\theta}(X)$ no longer apply when $G$ is infinite since the $\theta$-stability depends now on 
infinitely many parameters (one for each irreducible $G$-module), and it is not clear whether the $\theta$-stability condition can still be 
expressed as a GIT-stability condition (which depends only on a finite number of parameters). Nevertheless, a GIT-stability condition was 
defined in \cite[\S 2]{BT15}, depending on a finite subset $D \subset \Irr G$, and it was proved that for $D$ big enough in $\Irr G$, the 
$\theta$-stability of a $(G,h)$-constellation implies its GIT-stability. The possible converse implication and the relations between 
$\theta$-semistability and GIT-semistability were addressed in \cite[\S 5]{BT15} but these questions remained unanswered at that time. 
In this article, we will see that those two stability conditions are actually different. 

\begin{proposition**} [\S \ref{relations} and \S \ref{examples}]
The notions of $\theta$-(semi)stability and GIT-(semi)sta\-bility mentioned above for $(G,h)$-constellations do not coincide.
\end{proposition**}   

This proposition answers negatively \cite[Question 5.2]{BT15} and implies that the GIT approach used in \cite{BT15} to construct 
$M_{\theta}(X)$ is unsuitable to construct the moduli space of $\theta$-semistable $(G,h)$-constellations on $X$; in particular, the answer to \cite[Question 5.1]{BT15} is also negative. 

Once we know that those two stability conditions do not coincide, it is natural to compare them. The first step in this article 
is to reformulate the $\theta$-stability and the GIT-stability in terms of slope stability conditions, giving rise to the $\mut$-stability 
and the $\mu_D$-stability (where the index $D$ is to emphasize the dependence on $D$). The advantage of dealing with these new stability conditions 
defined by slopes is that we can then construct for any $(G,h)$-constellation the so-called \emph{Harder-Narasimhan filtration} \cite{HN75}. Within the 
years, the latter has been proved to be an extremely useful tool in the study of properties of moduli spaces in algebraic geometry.
The Harder-Narasimhan filtration is defined recursively by considering at each step the maximal destabilizing subobject; see \S \ref{HN1} 
for a precise definition. In some sense, this filtration measures how far an object is from being semistable. Therefore, comparing stability 
conditions reduces to comparing the corresponding Harder-Narasimhan filtrations for each object. In \S \ref{relations3}, we will 
explain how to associate to each Harder-Narasimhan filtration a polygon which encodes the numerical data of the filtration. The next statement 
gathers our results.  

\begin{theorem**}
Let $G$ be an infinite reductive algebraic group acting on an affine scheme of finite type $X$, let $h: \, \Irr G \to \NN_{\geq 0}$ be a Hilbert 
function, and let $\FF$ be a $(G,h)$-constellation on $X$. Let $D \subset \Irr G$ be a finite subset satisfying Hypothesis \ref{hypoD}. 
Then the following holds:
\begin{enumerate}[(i)]

\item $\FF$ admits a $\mut$-Harder-Narasimhan filtration $\FF_{\bullet}$ (Theorem \ref{HNexists}) as well as a $\mu_D$-Harder-Narasimhan filtration 
$\GG_{\bullet}^{D}$ (Theorem \ref{HNDexists}).  

\item If the finite subset $D \subset \Irr G$ is big enough, then $\FF_{\bullet}$ is a subfiltration of $\GG_{\bullet}^{D}$ (Theorem \ref{HNterms}).
Moreover, $\GG_{\bullet}^{D}$ is a subfiltration of some Jordan-H\"older filtration of the $\mut$-semistable factors
of $\FF$ (Remark \ref{HNtoJH}).

\item Even though the $\mu_D$-Harder-Narasimhan filtration might not stabilize when $D \subset \Irr G$ grows (\S \ref{examples}), 
the sequence of polygons associated with $\left(\GG_{\bullet}^{D}\right)_{D \subset \Irr G}$ converges to the polygon associated with 
$\FF_{\bullet}$ when $D$ grows (Theorem \ref{convergence}).
\end{enumerate}    
\end{theorem**}

The fact that the $(G,h)$-constellations we consider here (those generated in $D_{-}$, see Definition \ref{geninD-}) 
do not form an abelian category prevents us from applying the results of \cite{Rud:1997} to obtain 
directly the existence of the Harder-Narasimhan filtrations. Actually, we have to substantially modify the classical proofs for existence and 
uniqueness of Harder-Narasimhan filtrations in our situation.

The paper is organized as follows. In \S \ref{def_and_HN} we introduce $(G,h)$-constellations and $\theta$-stability. 
Then we convert the $\theta$-stability into a slope stability condition, the $\mut$-stability, and we prove the existence and 
uniqueness of the $\mut$-Harder-Narasimhan filtration for any $(G,h)$-constellation. 

Then, \S \ref{GITsection} is devoted to recall the construction of the moduli space of $\theta$-stable $(G,h)$-constellations as 
in \cite{BT15}. When performing this construction by using GIT, a GIT-stability condition is introduced, which depends on a choice of a finite subset 
$D \subset \Irr G$. Then, as for the $\theta$-stability, we convert this GIT-stability condition into a slope stability condition, 
the $\mu_{D}$-stability, and we construct the $\mu_{D}$-Harder-Narasimhan filtration. 

In \S \ref{comparisonstab}, which is the heart of this article, we study closely the relations between $\theta$-stability and GIT-stability.
First, we summarize the implications between the different stability notions considered in this article and answer related questions raised in \cite[\S 5]{BT15}. Then, once we know that $\theta$-stability and GIT-stability do not coincide for $(G,h)$-constellations, it is natural to compare the corresponding 
Harder-Narasimhan filtrations. We make explicit the relations between these two filtrations and prove parts (ii) and (iii) of Theorem B. 

Finally, \S \ref{examples} illustrates Proposition A and Theorem B by providing examples of the different phenomena that can occur. \\

\textbf{Notation.}
Throughout this article we work over the field of complex numbers $\CC$. Let $G$ be an infinite reductive algebraic group. Then we denote by $\Irr G$ 
the set of isomorphy classes of irreducible $G$-modules $\rho: G \to \GL(V_\rho)$, and by $R(G)=\bigoplus_{\rho \in \Irr G} \NN \cdot \rho$ 
the representation monoid of $G$. An element of $R(G)$ identifies naturally with a function $h: \Irr G \to \NN$; we call such a function a 
\emph{Hilbert function}. Let $X$ be an affine $G$-scheme of finite type. We say that $\FF$ is an \emph{$(\OO_X,G)$-module} if $\FF$ is a 
$G$-equivariant coherent $\OO_X$-module whose module of global sections $H^0(\FF)$ is a $G$-module with finite multiplicities. We denote the category of $(\OO_X,G)$-modules by $\Coh^G(X)$. We say that $h$ is the Hilbert function of $\FF$ if the multiplicities of the $G$-module 
$H^0(\FF)$ are given by $h$.

Whenever the word \emph{(semi)stable} appears in the text, or the abbreviation \emph{(s)s}, two statements should be read: a first one for \emph{stable}
or \emph{s}, and a second one for \emph{semistable} or \emph{ss}. If they appear together with the symbols $\sgeq$ or $\sleq$, one should read
$>$ or $<$ with \emph{stable}, and $\geq$ or $\leq$ with \emph{semistable}.

\section{Constellations and \texorpdfstring{$\mu_\theta$}{mu-theta}-Harder-Narasimhan filtration}  \label{def_and_HN}

We fix once and for all an (possibly non-connected) infinite reductive algebraic group $G$, an affine $G$-scheme of finite type $X$, and a non-zero Hilbert function $h: \Irr G \to \NN$.

\subsection{Constellations and \texorpdfstring{$\theta$}{theta}-stability}  \label{first}
In this subsection we present the notions of $(G,h)$-constellation, $\theta$-stability, and $\mut$-stability.

\begin{definition}
A $(G,h)$\textit{-constellation} on $X$ is an $(\OO_X,G)$-module $\FF$ such that 
$$H^0(\FF)=\bigoplus_{\rho \in \Irr G}\FF_{\rho}\otimes V_{\rho} \cong \bigoplus_{\rho \in \Irr G} V_{\rho}^{h(\rho)}$$
as a $G$-module, i.e., the multiplicities of the $G$-module $H^0(\FF)$ are given by the Hilbert function $h$.
\end{definition}

Moduli spaces parametrizing $(G,h)$-constellations are constructed in \cite{BT15}.
As the set of all $(G,h)$-constellations on $X$ is too large in general to be parametrized by a scheme,
the moduli problem is restricted to consider $(G,h)$-constellations 
satisfying a certain stability condition, the $\theta$-stability, that we now introduce. 

\begin{definition}  \label{deftheta}
Let $\theta=(\theta_\rho)_{\rho \in \Irr G}$ be a sequence of rational numbers (which depends on the Hilbert function $h$) satisfying:
\begin{itemize}
\item $\theta_\rho < 0$ for only finitely many $\rho \in \Irr G$;
\item $\theta_\rho > 0$ for infinitely many $\rho \in \Irr G$;
\item If $h(\rho)=0$, then $\theta_\rho=0$ ; and
\item $\sum_{\rho \in \Irr G} \theta_\rho h(\rho)=0$.
\end{itemize}
Then we call \emph{stability function} $\theta \colon R(G) \to \RR \cup \{\infty\}$ the function defined by
$$\theta(W) := \langle\theta,h_{W}\rangle := \sum_{\rho \in \Irr G} \theta_{\rho}\cdot \dim W_{\rho},$$
where $W = \bigoplus_{\rho \in \Irr G} W_{\rho}\otimes V_{\rho}$ is the isotypic decomposition of $W$.\\
In order to consider $\theta$ as a function $\Coh^G(X) \to \RR \cup \{\infty\}$, we set
$$\theta(\FF) := \theta(H^0(\FF)) = \sum_{\rho \in \Irr G} \theta_{\rho}\cdot \dim \FF_{\rho}.$$
In particular, if $\FF$ is a $(G,h)$-constellation, then we have
\begin{equation*}  
\theta(\FF) = \sum_{\rho \in \Irr G} \theta_{\rho}h(\rho)=0.
\end{equation*}
\end{definition}

The choice of $\theta$ induces a decomposition
\begin{equation}  \label{def_D}
\Irr G = D_+ \sqcup D_0 \sqcup D_- \qquad \text{such that} \qquad
\theta_{\rho}\left\{ \begin{array}{ll}
> 0 \text{ if} & \rho \in D_+ \\
= 0 \text{ if} & \rho \in D_0 \\
< 0 \text{ if} & \rho \in D_-
\end{array}\right.
\end{equation}
It follows from the definition of $\theta$ that $D_-$ is finite, $D_+$ is infinite, and the sets $\supp h \cap D_-$ and $\supp h \cap D_+$ are non-empty, where \emph{$\supp h:=\{\rho \in \Irr G \ |\ h(\rho) \neq 0\}$}.

\begin{definition}
\label{geninD-}
Let $\theta$ be as in Definition \ref{deftheta}, and let $\FF$ be an $(\OO_X,G)$-module.
If $\FF$ is generated by its negative part $\bigoplus_{\rho \in D_-} \FF_{\rho} \otimes V_{\rho}$
as an $\OO_X$-module, then we say that $\FF$ \textit{is generated in} $D_-$.
\end{definition}

Even though $(G,h)$-constellations need not be generated in $D_{-}$ in general, we will only consider in this article those generated in $D_-$. 
The reason to do that comes from the GIT construction of the moduli spaces of $(G,h)$-constellations in \cite{BT15}; this will become clear with the 
introduction of the invariant Quot scheme in \S \ref{2.1}.

\begin{warning}
From now on, $(G,h)$-constellations are always assumed to be generated in $D_-$ (with respect to a given $\theta$ as in Definition \ref{deftheta}). On the other hand, $(\OO_X,G)$-modules are not assumed to be generated in $D_-$, except when we say so.
\end{warning}

\begin{definition}  \label{mut}
Let $\FF$ be an $(\OO_X,G)$-module with a non-zero negative part, but not necessarily generated in $D_-$.
Then we define
$$r(\FF)=\sum_{\rho \in D_-} \dim \FF_\rho \in \NN_{>0},$$
and the \emph{$\theta$-slope} of $\FF$, where $\theta$ is a stability function as in Definition \ref{deftheta}, by
$$\mu_\theta(\FF):=\frac{-\theta(\FF)}{r(\FF)} \in \RR.$$
\end{definition}

\begin{remark}
The reason why we defined the rank of an $(\OO_X,G)$-module as in Definition \ref{mut} will become clear in view of Lemma \ref{noeth}.
\end{remark}

We now define two stability conditions on $(\OO_X,G)$-modules generated in $D_-$, which will turn out to be equivalent for $(G,h)$-constellations.

\begin{definition} \label{definitionstability}
Let $\theta$ be as in Definition \ref{deftheta}, and let $\FF$ be an $(\OO_X,G)$-module generated in $D_-$.
\begin{enumerate}
\item $\FF$ is called $\theta$\textit{-(semi)stable} if $\theta(\FF) = 0$ and for every $(\OO_X,G)$-submodule $0 \neq \FF' \subsetneq \FF$ generated in $D_-$ we have
$$\theta(\FF') \sgeq 0\; .$$
\item $\FF$ is called $\mu_\theta$\textit{-(semi)stable} if for all $(\OO_X,G)$-submodule $0 \neq \FF' \subsetneq \FF$ generated in $D_-$ we have
$$\mu_\theta(\FF') \sleq \mu_\theta(\FF)\;.$$
\end{enumerate}
If $\FF$ is non $\theta$-semistable resp. non $\mut$-semistable, we say that $\FF$ is \emph{$\theta$-unstable} resp. \emph{$\mut$-unstable}. 
\end{definition}

\begin{lemma}  \label{comparison}
Let $\theta$ be as in Definition \ref{deftheta}, and let $\FF$ be an $(\OO_X,G)$-module generated in $D_-$.
Then $\FF$ is $\theta$-(semi)stable if and only if $\FF$ is $\mu_\theta$-(semi)stable and $\theta(\FF)=0$. 
In particular, the notions of $\theta$-(semi)stability and $\mu_\theta$-(semi)stability are equivalent for $(G,h)$-constellations.
\end{lemma}
\begin{proof}
This follows immediately from Definition \ref{definitionstability}.
\end{proof}

The $\theta$-stability condition is the stability condition used in \cite{BT15} to prove 
the existence of a moduli space of stable $(G,h)$-constellations; see \cite[Theorem 4.3]{BT15}. 
The reason for considering $\mu_{\theta}$-stability instead of $\theta$-stability --which coincide for 
$(G,h)$-constellations by Lemma \ref{comparison}-- is that this reformulation in terms of slopes will allow us to talk about Harder-Narasimhan filtrations.

\subsection{\texorpdfstring{$\mu_\theta$}{mu-theta}-Harder-Narasimhan filtration} \label{HN1}

In this subsection, we construct the $\mut$-Harder-Narasimhan filtration for a $(G,h)$-constellation (Theorem \ref{HNexists}). 
We follow the classical treatment, see for instance \cite[\S 1.3]{HL10}, but conveniently adapted to our situation. We first show the existence of a unique maximal destabilizing subobject (Proposition \ref{unique}), and then we proceed by induction to prove the existence and uniqueness (which is a consequence of Proposition \ref{unik}) of the $\mut$-Harder-Narasimhan filtration. 

\begin{theorem} \label{HNexists}
Let $\mut$ be as in Definition \ref{mut}, and let $\FF$ be a $(G,h)$-constellation. 
Then $\FF$ has a unique filtration
$$0\subsetneq \FF_{1}\subsetneq \FF_{2}\subsetneq\cdots\subsetneq \FF_{t}\subsetneq \FF_{t+1}=\FF$$
verifying
\begin{enumerate}[(i)]
 \item each $\FF_i$ is an $(\OO_X,G)$-submodule generated in $D_-$;
 \item each quotient $\FF^{i}:=\FF_{i} / \FF_{i-1}$ is $\mut$-semistable; and
 \item the slopes of the quotients are strictly decreasing
 $$\mut(\FF^{1})>\mut(\FF^{2})>\cdots >\mut(\FF^{t})>\mut(\FF^{t+1}).$$
\end{enumerate}
We call this filtration the \emph{$\mu_\theta$-Harder-Narasimhan filtration} ($\mu_\theta$-HN filtration for short) of $\FF$. 
The integer $t+1$ is called the \emph{length} of the filtration.
\end{theorem}

Let us note that the $\mu_\theta$-HN filtration of a $(G,h)$-constellation $\FF$ is trivial if and only if $\FF$ is $\mut$-semistable. 
Explicit examples of $(G,h)$-constellations with non trivial $\mut$-HN filtration will be computed in \S \ref{examples}. 

The proof of Theorem \ref{HNexists} is postponed till the end of the subsection. It is clear from the definition that every $(\OO_X,G)$-module generated in $D_-$ is a $(G,\tilde h)$-constellation for a certain Hilbert function $\tilde h$. We will actually prove Theorem \ref{HNexists} for an arbitrary $(\OO_X,G)$-module generated in $D_-$ even though we formulate it for $(G,h)$-constellations which are the objects we are interested in. 

First, we prove a lemma --that we call \emph{seesaw property} following the terminology in \cite{Rud:1997}-- which relates slopes of objects in exact sequences. 

\begin{lemma} \emph{(seesaw property)} \label{ineq1}
Given a short exact sequence 
$$0\rightarrow \FF'\rightarrow \FF\rightarrow \FF''\rightarrow 0$$ 
of $(\OO_X,G)$-modules, 
all with non-zero negative part, we have
$$\mu_{\theta}(\FF') \leq \mu_{\theta}(\FF)\Longleftrightarrow \mu_{\theta}(\FF') \leq \mu_{\theta}(\FF'')\Longleftrightarrow
\mu_{\theta}(\FF) \leq \mu_{\theta}(\FF'')\; .$$
Moreover, if any of the inequalities is an equality, the other two are also equalities. 
\end{lemma}

\begin{proof}
Denote by $h'$, $h$, and $h''$ the Hilbert functions of $\FF'$, $\FF$, and $\FF''$ respectively. Since $G$ is a reductive group, it is clear that $h(\rho)=h'(\rho)+h''(\rho)$ for all $\rho\in \Irr G$. Then we have $\theta(\FF)=\theta(\FF')+\theta(\FF'')$ and
$r(\FF)= r(\FF')+r(\FF'')$. The assumption of having a non-zero negative part guarantees that $r(\FF)$, $r(\FF')$, and $r(\FF'')$ are non-zero.
It follows that
$$\mu_{\theta}(\FF)=\frac{-\theta(\FF)}{r(\FF)}=\frac{-\theta(\FF')-\theta(\FF'')}{r(\FF')+r(\FF'')}\; , $$
whence
$$-\theta(\FF)r(\FF')+\theta(\FF')r(\FF)=-\theta(\FF'')r(\FF)+\theta(\FF)r(\FF'')\; ,$$
which implies
$$-\theta(\FF)r(\FF')+\theta(\FF')r(\FF)\geq 0 \Longleftrightarrow -\theta(\FF'')r(\FF)+\theta(\FF)r(\FF'')\geq 0.$$
The last equivalence turns out to be
$$\mu_{\theta}(\FF') \leq \mu_{\theta}(\FF)\Longleftrightarrow
\mu_{\theta}(\FF) \leq \mu_{\theta}(\FF'')\; .$$
A similar treatment shows the equivalence of these two inequalities with the other one. 
Finally, we note that all implications still hold if we replace equalities by inequalities. 
\end{proof}

We now recall a result expressing the finiteness of the different functions which can appear as Hilbert functions of subobjects. 

\begin{proposition} \label{finiteness} \emph{(\cite[Prop. 1.9]{BT15})}
Let $\tilde h$ be an arbitrary Hilbert function. 
There is a finite set of Hilbert functions $\{h_1,\ldots,h_N\}$ such that for any $(G,\tilde h)$-constellation $\FF$ and any 
$(\OO_X,G)$-submodule $\FF' \subseteq \FF$ generated in $D_-$, the Hilbert function $h'$ of $\FF'$ is one of the $h_1,\ldots,h_N$.
\end{proposition}

Among the set of Hilbert functions, there exists a partial order defined by:
\begin{equation} \label{order}
h_1 \geq h_2 \Longleftrightarrow \forall \rho \in \Irr G,\ h_1(\rho) \geq h_2(\rho).
\end{equation}
Let us note the following basic facts which will be very useful in the remaining of this subsection: 
If there is an inclusion of $(\OO_X,G)$-modules $\FF_1 \subseteq \FF_2$ with Hilbert functions $h_1$ and $h_2$ respectively, then $h_1 \leq h_2$. 
Moreover, if $\FF_1 \subseteq \FF_2$, then $\FF_1=\FF_2$ if and only if $h_1=h_2$.  

The next result guarantees the existence of a unique maximal $\mu_{\theta}$-destabilizing subobject for $(\OO_X,G)$-modules generated in $D_-$.

\begin{proposition}  \label{unique}
Let $\FF$ be an $(\OO_X,G)$-module generated in $D_-$.
Then there exists a unique $(\OO_X,G)$-submodule $\FF' \subseteq \FF$ generated in $D_-$ such that:
\begin{enumerate}[(i)]
\item if $0 \neq \GG \subseteq \FF$ is generated in $D_-$, then $\mut(\GG) \leq \mut(\FF')$; and
\item if $0 \neq \GG \subseteq \FF$ is generated in $D_-$  and $\mut(\GG)=\mut(\FF')$, then $\GG \subseteq \FF'$.
\end{enumerate}
\end{proposition}

\begin{proof}
It is clear that if $\FF'$ exists then it has to be unique by (ii). Let us prove the existence of $\FF'$.

With the notation of Proposition \ref{finiteness} applied to $h$, let $M:=\max\{\mut(h_i)\}_{i=1}^{N}$.
Among the Hilbert functions satisfying $\mut(h_i)=M$, we pick one which is maximal for the partial order defined by \eqref{order}.
Let $h'$ be such a Hilbert function, and let $\FF'\subseteq \FF$ be an $(\OO_X,G)$-submodule generated in $D_-$ whose Hilbert function is $h'$.
Then by construction $\FF'$ satisfies (i). A priori $h'$ and $\FF'$ are not unique, however we will show that $\FF'$ satisfies (ii), and this will implies the uniqueness of $h'$ and $\FF'$.

Let us now prove that $\FF'$ satisfies (ii). Let $0 \neq \GG \subseteq \FF$ be an $(\OO_X,G)$-submodule generated 
in $D_-$ such that $\mut(\GG)=\mut(\FF')$.
Consider the exact sequence of $(\OO_X,G)$-modules:
\begin{equation} \label{hi}
0 \to \FF' \cap \GG \to \FF' \oplus \GG \to \FF' + \GG \to 0,
\end{equation}
where $\FF' \oplus \GG$ and $\FF'+\GG$ are generated in $D_-$ (since $\FF'$ and $\GG$ also are) but not necessarily $\FF' \cap \GG$.
We denote by $h_1$, $h_2$, $h_3$, and $h_4$ the Hilbert functions of $\FF'$, $\GG$, $\FF'+\GG$, and $\FF' \cap \GG$ respectively.
We distinguish between two cases.
\begin{enumerate}[(a)]
\item If the negative part of $\FF' \cap \GG$ is zero, i.e., if $h_4(\rho)=0$ for all $\rho \in D_-$. Then we deduce from \eqref{hi} that $h_1(\rho)+h_2(\rho)=h_3(\rho)$ for all $\rho \in D_-$. It follows that
\begin{itemize}
\item $r(\FF'+\GG)=r(h_3)=r(h_1+h_2)=r(\FF' \oplus \GG)$;
\item $\sum_{\rho \in D_-} \theta_\rho h_3(\rho)=\sum_{\rho \in D_-} \theta_\rho (h_1(\rho)+h_2(\rho))$; and
\item $\sum_{\rho \in D_+} \theta_\rho h_3(\rho) \leq \sum_{\rho \in D_+} \theta_\rho (h_1(\rho)+h_2(\rho))$.
\end{itemize}
Hence
$$ \mut(\FF'+\GG)=\frac{-\theta(\FF'+\GG)}{r(\FF'+\GG)} \geq \frac{-\theta(\FF' \oplus \GG)}{r(\FF' \oplus \GG)}=\mut(\FF' \oplus \GG).$$
Now since $\mut(\FF')=\mut(\GG)$, the seesaw property applied to 
$$ 0 \to \FF' \to \FF' \oplus \GG \to \GG \to 0$$
gives $\mut(\FF' \oplus \GG)=\mut(\FF')$.
Therefore $\FF'+\GG$ is an $(\OO_X,G)$-submodule of $\FF$ with greater $\mut$-slope and whose Hilbert function $h_3$ is greater or equal to $h_1$ for the partial order defined by \eqref{order}. By definition of $\FF'$, we must have $h_1=h_3$; this implies that $\FF'+\GG=\FF'$, i.e., that $\GG \subseteq \FF'$.

\item If the negative part of $\FF' \cap \GG$ is non-zero, then we denote by $\widetilde{\FF' \cap \GG}$ the 
$(\OO_X,G)$-submodule generated by its negative part; $\widetilde{\FF' \cap \GG}$ can be $\FF' \cap \GG$ itself or a non-zero proper subsheaf. 
We denote the Hilbert function of $\widetilde{\FF' \cap \GG}$ by $h_5$. 
By definition, we have $h_5(\rho)=h_4(\rho)$ for all $\rho \in D_-$, and $h_5(\rho) \leq h_4(\rho)$ for all $\rho \in \Irr G \setminus D_-$.
So arguing as before, we easily prove that
$$ \mut(\widetilde{\FF' \cap \GG}) \geq \mut(\FF' \cap \GG).$$
Suppose that $\GG \not\subset \FF'$, and consider the two exact sequences of non-zero $(\OO_X,G)$-modules:
\begin{equation} \label{h1}
0 \to \FF' \cap \GG \to  \GG \to \mathcal{U} \to 0,
\end{equation}
and
\begin{equation} \label{h2}
0 \to \FF'  \to \FF'+\GG \to \mathcal{U} \to 0.
\end{equation}
Since $\mut(\GG)=\mut(\FF')$ is maximal among $(\OO_X,G)$-submodules of $\FF$ generated in $D_-$, we necessarily have $\mut(\GG) \geq \mut(\widetilde{\FF' \cap \GG})$. On the other hand, we just saw that $\mut(\widetilde{\FF' \cap \GG}) \geq \mut(\FF' \cap \GG)$, hence $\mut(\GG) \geq \mut(\FF' \cap \GG)$. 
The seesaw property applied to \eqref{h1} gives $\mut(\GG) \leq \mut(\mathcal{U})$. 
Thus, since $\mut(\FF')=\mut(\GG)$, the seesaw property applied to \eqref{h2} gives $\mut(\FF') \leq \mut(\FF'+\GG)$. By definition of $\FF'$, this 
implies that $\FF'+\GG=\FF'$, which contradicts our assumption $\GG \not\subset \FF'$. 
\end{enumerate}
Therefore, if $0 \neq \GG \subseteq \FF$ is an $(\OO_X,G)$-submodule generated in $D_-$ such that $\mut(\GG)=\mut(\FF')$, then necessarily $\GG \subseteq \FF'$.
\end{proof}

The next proposition assures the uniqueness of the first term of the $\mut$-HN filtration.  

\begin{proposition}  \label{unik}
Let $\FF$ be an $(\OO_X,G)$-module generated in $D_-$ with a filtration satisfying the properties (i)-(iii) of Theorem \ref{HNexists}. Then the first term $\FF_1$ of the filtration is the $(\OO_X,G)$-submodule $\FF'$ given by Proposition \ref{unique}.
\end{proposition}

\begin{proof}
The proof goes by induction on the length $t+1$ of the filtration. If $t=0$, then $\FF$ is $\mu_{\theta}$-semistable and $\FF_1=\FF=\FF'$.
We now suppose that $t \geq 1$, and we consider the filtration of length $t$ given by
$$0 \subsetneq \FF_{2}/\FF_{1}\subsetneq \FF_3/\FF_1 \subsetneq \cdots\subsetneq \FF_{t}/\FF_1\subsetneq \FF_{t+1}/\FF_1=\FF/\FF_1.$$
Then it is clear that properties (i)-(iii) of Theorem \ref{HNexists} are again satisfied.  
Hence, by induction hypothesis, we know that $\FF_{2}/\FF_{1}=(\FF/\FF_1)'$. We want to deduce from this that $\FF_1=\FF'$, i.e., that $\FF_1$ satisfies the properties (i) and (ii) of Proposition \ref{unique}. Let $0 \neq \GG \subseteq \FF$ be an $(\OO_X,G)$-submodule generated in $D_-$. 
We distinguish between several cases:

\begin{enumerate}[a)]
\item If $\GG \subset \FF_1$, then $\mut(\GG) \leq \mut(\FF_1)$ since $\FF_1$ is $\mut$-semistable. 

\item If $\GG \not\subseteq \FF_1$, then $\GG/(\FF_1 \cap \GG)$ is a non-zero $(\OO_X,G)$-submodule of $\FF /\FF_1$ generated in $D_-$, and thus 
$$\mut(\GG/(\FF_1 \cap \GG)) \leq \mut(\FF_{2}/\FF_{1})<\mut(\FF_1),$$ 
where the first inequality is by induction hypothesis and the second inequality is property (iii) of Theorem \ref{HNexists}. We 
denote by $\widetilde{\FF_1 \cap \GG}$ the $(\OO_X,G)$-submodule generated by the negative part of $\FF_1 \cap \GG$.  

\begin{itemize}
\item If $\widetilde{\FF_1 \cap \GG}=0$, then an explicit calculation gives $\mut(\GG) \leq \mut(\GG/(\FF_1 \cap \GG))$, and 
thus $\mut(\GG)< \mut(\FF_1)$.  

\item If $\widetilde{\FF_1 \cap \GG}\neq 0$, then one easily checks that $\mut(\FF_1 \cap \GG) \leq \mut(\widetilde{\FF_1 \cap \GG})$. 
Since $\FF_1$ is $\mut$-semistable, we have $\mut(\widetilde{\FF_1 \cap \GG}) \leq \mut(\FF_1)$, and thus $\mut(\FF_1 \cap \GG) \leq \mut(\FF_1)$.  
The seesaw property applied to 
\begin{equation} \label{h3}
0 \to \FF_1 \cap \GG \to \GG \to \GG/(\FF_1 \cap \GG) \to 0
\end{equation}
implies that either 
$$\mut(\GG) \leq \mut(\GG/(\FF_1 \cap \GG)) < \mut(\FF_1)$$ or $$\mut(\GG)<\mut(\FF_1 \cap \GG) \leq \mut(\FF_1).$$
\end{itemize}
\end{enumerate}
In all cases, we get that $\mut(\GG) \leq \mut(\FF_1)$, i.e, that $\FF_1$ satisfies (i) of Proposition \ref{unique}.

Let us now suppose that $\GG$ satisfies $\mut(\GG)=\mut(\FF_1)$ but is not contained in 
$\FF_1$. We have seen in b) that either $\widetilde{\FF_1 \cap \GG}=0$, and then 
$$\mut(\GG) \leq \mut(\GG/(\FF_1 \cap \GG))<\mut(\FF_{1})$$
which is a contradiction, or else $\widetilde{\FF_1 \cap \GG} \neq 0$, and then 
$$\mut(\FF_1 \cap \GG) \leq \mut(\FF_1)=\mut(\GG).$$ 
In the second case, 
the seesaw property applied to \eqref{h3} gives $\mut(\GG) \leq \mut(\GG/(\FF_1 \cap \GG))$. But $\mut(\FF_2/\FF_1)<\mut(\FF_1)=\mut(\GG)$, 
hence $\mut(\FF_2/\FF_1)<\mut(\GG/(\FF_1 \cap \GG))$, which contradicts our induction assumption. Therefore $\GG \subseteq \FF_1$, and 
thus $\FF_1$ satisfies (ii) of Proposition \ref{unique}. Then the result follows from the uniqueness of an $(\OO_X,G)$-submodule of $\FF$ generated in $D_-$ and 
satisfying properties (i) and (ii) of Proposition \ref{unique}.
\end{proof}

\begin{proof}[Proof of Theorem \ref{HNexists}]
Let us prove the existence of a filtration satisfying properties (i)-(iii) of Theorem \ref{HNexists} by induction on the dimension of the negative part of $\FF$. If $\FF$ is a $\mut$-semistable $(\OO_X,G)$-module, then the filtration $0 \subsetneq \FF$ satisfies (i)-(iii). Otherwise, let $\FF'$ be the $(\OO_X,G)$-submodule given by Proposition \ref{unique}. Then $0<r(\FF/\FF')<r(\FF)$, and thus by induction hypothesis, there exists a filtration
$$0\subsetneq \overline{\FF}_1 \subsetneq \overline{\FF}_2\subsetneq\cdots\subsetneq \overline{\FF}_{t-1}\subsetneq \overline{\FF}_t=\FF/\FF',$$
for some $t \geq 1$, verifying the assumptions (i)-(iii) of Theorem \ref{HNexists}.
For $i \geq 2$, we denote by $\FF_i$ the preimage of $\overline{\FF}_{i-1}$ in $\FF$ and we denote $\FF_1:=\FF'$. 
Then one easily checks that the filtration
$$0\subsetneq \FF_{1}\subsetneq \FF_{2}\subsetneq\cdots\subsetneq \FF_{t}\subsetneq \FF_{t+1}=\FF$$
also satisfies properties (i)-(iii) of Theorem \ref{HNexists}.

It remains to prove the uniqueness part of Theorem \ref{HNexists}, but this is a direct consequence of Proposition \ref{unik}. 
\end{proof}

\begin{remark} \label{JH}
Using the same arguments as for Proposition \ref{unique}, we can prove that 
every $\mut$-semistable $(G,h)$-constellation $\FF$ has a (generally non-unique) \emph{$\mut$-Jordan-H\"older filtration}, i.e., a filtration 
$$0\subsetneq \JJ_{1}\subsetneq \JJ_{2}\subsetneq\cdots\subsetneq \JJ_{s}\subsetneq \JJ_{s+1}=\FF$$
verifying that all $\JJ_{i}/\JJ_{i-1}$ are $\mut$-stable and $\mut(\JJ_{1})=\mut(\JJ_{2})= \cdots =\mut(\JJ_{s+1})$. 
Then the graded object $\bigoplus_{i=1}^{s+1}\JJ_{i}/\JJ_{i-1}$ is unique, i.e., it does not depend (up to isomorphism) on the choice of the $\mut$-Jordan-H\"older filtration. 
\end{remark}

\section{GIT-stability and \texorpdfstring{$\mu_D$}{mu-D}-Harder-Narasimhan filtration}  \label{GITsection}

In this section we introduce the notions of GIT-stability and $\mu_D$-stability for $(G,h)$-constellations.
In \S \ref{2.1} we introduce the invariant Quot scheme $\GQuot$, then in \S \ref{origin} we explain how to 
identify the $(G,h)$-constellations with certain elements of $\GQuot$. In particular, we will see that isomorphy classes of $(G,h)$-constellations are in one-to-one correspondence with certain $\Gamma$-orbits of $\GQuot$, where $\Gamma$ is a reductive algebraic group acting on $\GQuot$ defined in \S \ref{gamma1}. 
It is then natural in \S\S \ref{gamma1}--\ref{gamma2} to consider the GIT-quotient of $\GQuot$ by the $\Gamma$-action, and that 
is how the GIT-stability comes into the picture. 
Indeed, the invariant Quot scheme $\GQuot$ being quasi-projective, we need to restrict the $\Gamma$-action to the open subset of
GIT-(semi)stable points $\GQuot^{(s)s}$ to obtain a categorical quotient. The correspondence between elements of $\GQuot$ and 
$(G,h)$-constellations, in turn, allows us to talk about GIT-(semi)stable $(G,h)$-constellations. In \S \ref{HN2} we introduce a 
new stability condition on the $(\OO_X,G)$-modules generated in $D_-$, the $\mu_D$-stability, which is a slope stability condition. Finally, we prove that $\mu_D$-stability and GIT-stability coincide for $(G,h)$-constellations, and we construct the $\mu_D$-Harder-Narasimhan filtration associated with a $(G,h)$-constellation. 

Let us mention that \S\S \ref{2.1}--\ref{gamma2} are mainly extracted from \cite[\S 2 and \S 3]{BT15}, but \S \ref{HN2}, which is the most important part of this section, is an original work.
 
As before, we fix a Hilbert function $h: \Irr G \to \NN$ and a stability function $\theta$; see Definition \ref{deftheta}. 
The classical reference for the concepts related to Geometric Invariant Theory is \cite{GIT:1994}.

\subsection{The invariant Quot scheme} \label{2.1}

For every $\rho \in D_-$, let $A_\rho=\CC^{h(\rho)}$. We define the $G$-equivariant free $\OO_X$-module of finite rank
\begin{equation}\label{defH}
\HH := \left( \bigoplus_{\rho \in D_-} A_\rho \otimes V_{\rho} \right) \otimes \OO_X,
\end{equation}
and we denote by $\GQuot$ the \emph{invariant Quot scheme} which parametrizes all the
$(\OO_X,G)$-submodules $\KK \subseteq \HH$ such that $\HH/\KK$ is a $(G,h)$-constellation. Equivalently, 
the invariant Quot scheme parametrizes the equivalence classes of quotient maps $[q\colon \HH \twoheadrightarrow \FF]$, 
where $\FF$ is a $(G,h)$-constellation; two quotients $q$ and $q'$ being in the same equivalence class if $\Ker q= \Ker q'$.

The invariant Quot scheme was constructed by Jansou in \cite{Jan:2006}, and then used in \cite{BT15} to construct the 
moduli space of $\theta$-stable $(G,h)$-constellations. In the next subsection, we will explain how to associate a given 
$(G,h)$-constellation $\FF$ (generated in $D_-$ by assumption, see \S \ref{first}) with a quotient $[q:\HH \twoheadrightarrow \FF] \in \GQuot$. 
Let us emphasize that there is not such a correspondence for $(G,h)$-constellations not generated in $D_-$, and this is the reason why we 
consider only $(G,h)$-constellations generated in $D_{-}$ in this article.

\subsection{Quotients originating from a constellation} \label{origin}
Let $\FF$ be a $(G,h)$-constella\-tion, and let $H^0(\FF) = \bigoplus_{\rho \in \Irr G}\FF_{\rho}\otimes V_{\rho}$ be the isotypic decomposition of 
its space of global sections. Since $\FF_{\rho} = \mathcal{H}om_G(V_{\rho}, H^0(\FF))$, we have evaluation maps
\begin{equation} \label{evv}
ev_{\rho}\colon \FF_{\rho} \otimes V_{\rho} \otimes H^0(\OO_X) \to H^0(\FF),\ \;\alpha \otimes v \otimes f \mapsto f \cdot \alpha(v),
\end{equation}
and $H^0(\FF)$ is generated as an $H^0(\OO_X)$-module by the images of $ev_{\rho}$ ($\rho \in D_-$) by assumption. 
We choose a basis of each $\FF_{\rho}$, i.e., we fix an isomorphism $\psi_{\rho}\colon A_{\rho} \to \FF_{\rho}$, and we compose it with the evaluation map \eqref{evv} considered as a map between $\OO_X$-modules. We obtain
\begin{equation}\label{quotofconst}
q_{\rho}\colon A_{\rho} \otimes V_{\rho} \otimes \OO_X \to \FF,\ \;a \otimes v \otimes f \mapsto f \cdot \psi_{\rho}(a)(v).
\end{equation}
Their sum
$$q := \underset{\rho \in D_-}{\oplus} q_{\rho}\colon \HH = \bigoplus_{\rho \in D_-} A_{\rho} \otimes V_{\rho} \otimes \OO_X \to \FF$$
gives us a point $[q\colon \HH \twoheadrightarrow \FF] \in \GQuot$ with the property that the map
\begin{equation} \label{iso} \varphi_{\rho}\colon A_{\rho} \to \FF_{\rho} = \mathcal{H}om_G(V_{\rho},H^0(\FF)),\ \; a \mapsto (v \mapsto q(a \otimes v \otimes 1)), \end{equation}
is just the isomorphism $\psi_{\rho}$ since, for $a \in A_{\rho}$ and $v \in V_{\rho}$, we have
$$\varphi_{\rho}(a)(v) = q(a \otimes v \otimes 1) = 1 \cdot \psi_{\rho}(a)(v) = \psi_{\rho}(a)(v).$$

\begin{definition} \label{origin2}
Let $[q:\HH \twoheadrightarrow \FF]$ be an element of the invariant Quot scheme $\GQuot$. If for every $\rho \in D_-$ the map $\varphi_\rho$ defined by \eqref{iso} is an isomorphism, then we say that $q$ \emph{originates from} $\FF$. 
\end{definition}

Different choices of bases for $\FF_\rho$ give different elements of $\GQuot$, and it is precisely to cancel this ambiguity that we will introduce 
in \S \ref{gamma1} the action of the group $\Gamma$ on $\GQuot$.

Conversely, given an element $[q\colon \HH \twoheadrightarrow \FF] \in \GQuot$, the quotient $\FF$ is a $(G,h)$--constellation. However, 
the induced maps $\varphi_{\rho}$ need not to be isomorphisms so that $[q]$ need not to originate from $\FF$ as above.

\begin{lemma}  \label{omega}
With the notation above, the subset
$$\Omega^G(\HH,h):=\{\ [q:\HH \twoheadrightarrow \FF] \in \GQuot\ |\ [q] \text{ originates from } \FF\ \}$$
is open in the invariant Quot scheme $\GQuot$.
\end{lemma}

\begin{proof}
By the discussion above, we have
$$\Omega^G(\HH,h)= \bigcap_{\rho \in D_-} \left\{\; [q] \in \GQuot \ |\ \rk \varphi_\rho=h(\rho) \;\right\},$$
where $\varphi_\rho$ is the linear map defined by \eqref{iso}. Since $\rk \varphi_\rho=h(\rho)$ is an open condition for each $\rho \in D_-$, we obtain the result.
\end{proof}

\begin{remark}
As $\GQuot$ is reducible in general, $\Omega^G(\HH,h)$ might not be a dense open subset. 
\end{remark}

\subsection{GIT setting}  \label{gamma1}

Consider the natural action of the group $\Gamma':= \prod_{\rho \in D_-} \GL(A_{\rho})$ on $\HH$ by
multiplication from the left on the constituent components. This action induces an action on $\GQuot$ from the right, which we describe.
Let $\gamma = (\gamma_{\rho})_{\rho \in D_-} \in \Gamma'$ and $[q\colon \HH \twoheadrightarrow \FF] \in \GQuot$.
Then $[q]\cdot\gamma$ is the map
$$[q] \cdot \gamma\colon \HH \twoheadrightarrow \FF, \ \; a_{\rho} \otimes v_{\rho} \otimes f \mapsto q(\gamma_{\rho}a_{\rho} \otimes v_{\rho} \otimes f).$$
As the subgroup of scalar matrices $K:=\{ \prod_{\rho \in D_-} \alpha  \mathrm{Id}_{A_\rho}; \alpha \in \CC^* \} \cong \CC^*$  acts trivially on $\GQuot$, we restrict to consider the action of the subgroup
\begin{equation} \label{def Gamma}
\Gamma:= \left \{ (\gamma_\rho)_{\rho \in D_-} \in \prod_{\rho \in D_-} \GL(A_\rho) \; \middle| \; \prod_{\rho \in D_-} \det(\gamma_\rho)=1 \right \}\; ,
\end{equation}
an action with finite stabilizers. 

From the correspondence between quotients and constellations explained in \S \ref{origin}, it is clear that the open subscheme 
$\Omega^G(\HH,h)$ defined in Lemma \ref{omega} is $\Gamma$-stable, and that there is a one-to-one 
correspondence between the $\Gamma$-orbits in $\Omega^G(\HH,h)$ and the isomorphy classes of $(G,h)$-constellations. 
Therefore, we are naturally interested in performing the GIT quotient of $\GQuot$ by $\Gamma$. 
However, to construct such a quotient we first need to fix a $\Gamma$-linearized ample line bundle on $\GQuot$; 
the latter will depend on a finite subset $D\subset \Irr G$.

\begin{proposition}  \label{embedding_etha} \emph{(\cite[\S 2.1]{BT15})}
There exists a finite subset $D\subset \Irr G$ (depending on $h$ and $\theta$)  
and, for each $\rho \in D$, a finite dimensional vector space $H_\rho$ such that there is a locally closed immersion
\begin{equation}\label{Quotimmersion}
\eta \colon \GQuot \hookrightarrow \prod\limits_{\rho \in D} \PP(\Lambda^{h(\rho)} H_{\rho}).
\end{equation}
\end{proposition}

Let us note that if $h(\rho)=0$ for any $\rho \in D$, 
then $\PP(\Lambda^{h(\rho)} H_{\rho})$ is a point, and thus $\rho$ plays no role in the embedding \eqref{Quotimmersion}. 
Therefore, we can assume that $h(\rho)\neq 0$ for all $\rho \in D$. 
Also, as noticed in \cite[Remark 2.2]{BT15}, for any set $D'$ containing $D$ we again obtain an embedding of $\GQuot$. Hence, adding further 
representations if necessary, we will always assume that the following hypothesis holds.

\begin{hypothesis}  \label{hypoD}
$D$ is a finite subset of $\Irr G$ such that the morphism \eqref{Quotimmersion} is a closed immersion, $D$ contains $D_-$ and intersects $D_+$, and $h(\rho) \neq 0$ for every $\rho \in D$ (i.e., $D$ is contained in $\supp h$).
\end{hypothesis}

Choose a sequence of positive integers $(\kappa_\rho)_{\rho \in D} \in \NN_{>0}^D$ and consider the ample line
bundles $\OO_{\rho}(\kappa_\rho)$ on
$\PP(\Lambda^{h(\rho)}H_{\rho})$ which together give an ample line bundle
\begin{equation}\label{linebdl}
\Ll = \eta^* \left( \bigotimes_{\rho \in D} \OO_{\rho}(\kappa_\rho) \right)
\end{equation}
on $\GQuot$. 
Let us note that the definition of $\Ll$ does not depend on the choice of the vector spaces $H_\rho$ in Proposition \ref{embedding_etha}; this follows namely from the explicit construction of the $H_\rho$ in the proof of \cite[Proposition 2.1]{BT15}.

The action of $\Gamma$ on $\GQuot$ induces a natural linearization on
some power $\Ll^k$ of $\Ll$; see the remark after \cite[Lemma 4.3.2]{HL10}. 
Replacing $\kappa_{\rho}$ by $k\kappa_{\rho}$ for each $\rho \in D$, we can assume that $\Ll$
itself carries a $\Gamma$-linearization.  

Further, let $\chi:\Gamma \to \CC^*$ be a character of $\Gamma$. Then
$\chi(\gamma) = \prod_{\rho \in D_-} \det(\gamma_\rho)^{\chi_\rho}$ with $(\chi_\rho)_{\rho \in D_-} \in \ZZ^{D_-}$.
We write $\Ll_{\chi}$ for the ample line bundle $\Ll$ equipped with the linearization twisted by the character $\chi$; this ample 
line bundle depends on $D$ by construction. Finally, we denote by $\GQuot_D^{(s)s}$ the open subset of GIT-(semi)stable points of 
$\GQuot$ with respect to $\Ll_\chi$; see \cite[Definition 1.7]{GIT:1994} for the definition of GIT-(semi)stable points. One 
should really keep in mind that $\GQuot_D^{(s)s}$ does depend on $D$ and that different choices of $D$ lead to different sets of GIT-(semi)stable points.

\begin{remark}
In the following, we will consider $\Ll_\chi$ with 
$\kappa_{\rho} \in \QQ_{>0}$ ($\rho \in D$) and $\chi_\rho \in \QQ$ ($\rho \in D_-$). 
In that case, it has to be understood that we replace each $\kappa_\rho$ by $p_1 \kappa_\rho$ and each 
$\chi_\rho$ by $p_2 \chi_\rho$, where $p_1$ resp. $p_2$, is the least common multiple of the denominators of all 
the $\kappa_{\rho}$ resp. of all the $\chi_{\rho}$.
\end{remark}

\subsection{Choice of GIT parameters}  \label{gamma2}

In this subsection, we fix the values of the GIT parameters $\kappa_\rho \;(\rho \in D)$ and $\chi_\rho \;(\rho \in D_-)$ 
in order to relate the GIT-(semi)stability for points of $\GQuot$ with the $\mu_\theta$-(semi)stability introduced 
for $(G,h)$-constel\-lations in \S \ref{first}; see Theorem \ref{relation} for a precise statement. 
Let us mention that the numerical values given in \cite[\S 3.3]{BT15} are not correct and should be replaced by the numerical values given below; see \cite{BT17} for more details.

Recall that we fixed a stability function $\theta$ at the beginning of \S \ref{GITsection}, and let $D$ be a finite subset of 
$\Irr G$ satisfying Hypothesis \ref{hypoD}.
We denote:
\begin{align*}
A &:=\bigoplus_{\rho \in D_-} A_\rho, \text{ which is a vector space of dimension } r(h)=\sum_{\rho \in D_-} h(\rho)  \;;\\
d &:= \# (D\backslash D_{-}) \in \NN_{>0}\; ; \text{ and}\\
S_D &:= \sum_{\rho \in \Irr G \setminus D}\theta_{\rho}h(\rho) \in \QQ_{>0}.
\end{align*}
Given numbers $\kappa_\rho$ ($\rho \in D$) and $\chi_\rho$ ($\rho \in D_-$), and any Hilbert function $h': \Irr G \to \NN$, we denote
\begin{align*}
\kappa_D(h')&:=\sum_{\rho \in D} \kappa_\rho h'(\rho) \;; \text{ and } \\
    \chi(h')&:=\sum_{\rho \in D_-} \chi_\rho h'(\rho) \;;
\end{align*}
where we stress the fact that the parameters $\kappa_\rho$ ($\rho \in D$) depend on $D$. 
We now fix the following values for the $\kappa_\rho$ ($\rho \in D$) and the $\chi_\rho$ ($\rho \in D_-$) introduced in \S \ref{gamma1}:
\begin{equation} \label{choicechikappa}
\left\{
    \begin{array}{llll}
\kappa_{\rho} &\in \QQ_{>0}  &&\text{for } \rho \in D_-\\
\kappa_{\rho} &= \theta_{\rho} + \frac{S_D}{d \cdot h(\rho)} &&\text{for } \rho \in D\backslash D_-\\
\chi_{\rho} &= \theta_{\rho} - \kappa_{\rho} + \frac{\kappa(h)}{r(h)}  &&\text{for } \rho \in D_-
\end{array}
\right.
\end{equation}
We recall that if $\rho \in D$, then $h(\rho) \neq 0$ by Hypothesis \ref{hypoD}. Let us also note that
$$\kappa(h):=\kappa_{D}(h)=\sum_{\rho \in D_-} \kappa_\rho h(\rho) + \sum_{\rho \in D_+} \theta_\rho h(\rho)$$
is well-defined, i.e., it does not depend on $D$.
Moreover, $S_D= - \sum_{\rho \in D} \theta_{\rho}h(\rho) \in \QQ_{>0}$, so $\kappa_{\rho} \in \QQ_{>0}$ for all $\rho \in D$.

Since $(G,h)$-constellations identify with elements of $\Omega^G(\HH,h)$ (see \S \ref{origin}), it makes sense to talk about $\mu_\theta$-(semi)stability for quotients. We define
\begin{equation*}  
\Omega^G(\HH,h)_{\theta}^{(s)s}:= \{\ [q:\HH \twoheadrightarrow \FF] \in \Omega^G(\HH,h)\;|\; \FF \text{ is $\mu_\theta$-(semi)stable}\ \}.
\end{equation*}
The latter is an open subscheme of $\Omega^G(\HH,h)$ by Lemma \ref{comparison} and \cite[\S 4.1]{BT15}.
The next result motivates the choices we made for the values of $\kappa_\rho$ and $\chi_\rho$.

\begin{theorem} \label{relation}
With the notation above and the GIT parameters given by \eqref{choicechikappa}, there exists a finite subset $D \subset \Irr G$ \emph{big enough} (i.e., $D$ contains a given $\tilde D$ and satisfies Hypothesis \ref{hypoD}), such that
$$\Omega^G(\HH,h)_{\theta}^{s} \subseteq \GQuot_{D}^{s} \subseteq \GQuot_{D}^{ss} \subseteq \Omega^G(\HH,h).$$
Moreover, each of these sets is $\Gamma$-stable and open in $\GQuot$.
\end{theorem}

\begin{proof}
The first inclusion is \cite[Theorem 3.10]{BT15}, the third inclusion is \cite[Lemma 3.1]{BT15}, and the last statement follows from the 
definition of GIT-(semi)stability and from \cite[Proposition 4.1]{BT15}.
\end{proof}

Since the set of GIT-(semi)stable points of $\Omega^G(\HH,h)$ is stable under the action of $\Gamma$, it makes sense to talk about 
GIT-(semi)stable $(G,h)$-constellations instead of GIT-(semi)stable quotients. Indeed, if $\FF$ is a $(G,h)$-constellation and there 
exist isomorphisms $\psi_{\rho}\colon A_{\rho} \to \FF_{\rho}$ (see the beginning of \S \ref{origin}) such that the 
corresponding quotient $[q:\HH \twoheadrightarrow \FF]$ is GIT-(semi)stable, then the same is true for any other choice of isomorphisms 
(by $\Gamma$-stability of the set of GIT-(semi)stable points).

Therefore, the GIT-stability can be seen as a stability condition on $(G,h)$-constellations. 
A set-theoretical version of Theorem \ref{relation} is given by
\begin{corollary}
Let $\mut$ be the stability condition of Definition \ref{definitionstability}. Then for any finite subset $D \subset \Irr G$ big enough, we have the inclusions 
\begin{small}
\begin{align*}
\left\{\begin{array}{c}\mu_\theta\text{--stable }\\(G,h)\text{--constellations}\end{array}\right\}
\subseteq \left\{\begin{array}{c}\text{GIT--stable }\\(G,h)\text{--constellations} \end{array}\right\}
\subseteq \left\{\begin{array}{c}\text{GIT--semistable }\\(G,h)\text{--constellations} \end{array}\right\} \; .
\end{align*}
\end{small}
\end{corollary}

\subsection{\texorpdfstring{$\mu_D$}{mu-D}-stability and \texorpdfstring{$\mu_D$}{mu-D}-Harder-Narasimhan filtration} \label{HN2}

In this subsection we introduce a new stability condition on the $(\OO_X,G)$-modules generated in $D_-$, the $\mu_D$-stability, 
which will be proved to coincide with the GIT-stability for $(G,h)$-constellations (Corollary \ref{muDisGIT}). 
This reformulation of the GIT-stability in terms of the slope $\mu_{D}$ will ultimately allow us to construct another Harder-Narasimhan filtration for $(G,h)$-constellations (Theorem \ref{HNDexists}). 

We fix a finite subset $D \subset \Irr G$ satisfying Hypothesis \ref{hypoD} and we keep the notation introduced in \S\S \ref{2.1}--\ref{gamma2}. 

\begin{lemma} \emph{(\cite[\S 2.3]{BT15})}
Let $[q:\HH \twoheadrightarrow \FF] \in \GQuot$ and let $\lambda: \CC^* \to \Gamma$ be a $1$-parameter subgroup. Then 
$$[\overline q]:=\lim_{t \to \infty} [q]\cdot \lambda(t)$$ 
is a well-defined element of $\GQuot$, which 
is a fixed point for the action of $\lambda$. In particular, $\lambda$ acts linearly on the fibre $\Ll_\chi(\overline q)$, where $\Ll_\chi$ is the $\Gamma$-linearized ample line bundle on $\GQuot$ defined at the end of \S \ref{gamma1}.
\end{lemma}

\begin{definition}
Given $[q \colon \HH \twoheadrightarrow \FF] \in \GQuot$ 
and  a $1$-parameter subgroup $\lambda$ of $\Gamma$, we denote by $\mu_{\Ll_{\chi}}(q,\lambda)$ the weight for the action of $\lambda$ on the 
fiber $\Ll_{\chi}([\overline{q}])$.
\end{definition}

In our situation, the Hilbert-Mumford numerical criterion \cite[Theorem 2.1]{GIT:1994} can be formulated as follows:

\begin{theorem}\label{HMcriterion} 
The point $[q \colon \HH \twoheadrightarrow \FF] \in \GQuot$ is GIT-(semi)stable if and only if $\mu_{\Ll_{\chi}}(q,\lambda) \sgeq 0$ for all non-trivial $1$-parameter subgroups $\lambda\colon \CC^* \to \Gamma$. 
\end{theorem}

After computing the weight $\mu_{\Ll_{\chi}}(q,\lambda)$ in terms of the GIT parameters of \S \ref{gamma2}, we can rewrite the Hilbert-Mumford 
numerical criterion as follows:

\begin{proposition}\label{HMcalculation} \emph{(\cite[Proposition 2.11]{BT15})}
The point $[q:\HH \twoheadrightarrow \FF]\in\GQuot$ is GIT-(semi)stable if
and only if for all graded subspaces $0\neq A'\subsetneq A$, that is $A'=\oplus_{\rho\in D_{-}}A_{\rho}'$ with $A_{\rho}'\subseteq A_{\rho}$ 
for every $\rho \in D_-$, the inequality
\begin{equation} \label{ek}
\dim A\cdot (\kappa_{D}(\FF')+\chi(A'))-\dim A'\cdot \kappa(h)\sgeq 0
\end{equation}
holds, where $\FF':=q\left(\oplus_{\rho\in D_{-}}A_{\rho}'\otimes V_{\rho}\otimes \OO_{X}\right)$, 
and $A$, $\kappa_{D}$, and $\chi$ are defined in \S \ref{gamma2}.
\end{proposition}

We follow the notation of \S \ref{origin}. If $[q: \HH \twoheadrightarrow \FF]\in \Omega^G(\HH,h)$, then we
may establish a correspondence between subsheaves of $\FF$ generated in $D_-$ and certain
graded subspaces of $A$. Let  $A' \subseteq A$ be a graded subspace, and let 
\begin{equation}\label{FausA}
\FF' := q \left( \bigoplus_{\rho \in D_-} A_{\rho}' \otimes V_{\rho} \otimes \OO_X \right) =  \OO_X \cdot \left( \sum_{\rho \in D_-} \varphi_{\rho}(A_{\rho}')(V_\rho)\right)
\end{equation}
be the $(\OO_X,G)$-submodule of $\FF$ generated by the $\varphi_{\rho}(A_{\rho}')$. 
Since $\varphi_{\rho}|_{A_{\rho}'}$ is injective, we have $\dim A_{\rho}' \leq \dim \FF_{\rho}'$ for every $\rho \in D_-$.
Now we define
$$\widetilde A_{\rho}' := \varphi_{\rho}^{-1}(\FF_{\rho}') \qquad \text{\ \ \ and\ \ \ } \qquad \widetilde A' := \bigoplus_{\rho \in D_-} \widetilde A_{\rho}'\ .$$
Roughly speaking, $\widetilde A'$ is the biggest graded subspace of $A$ which generates $\FF'$. 
For this reason, we call $\widetilde{A'}\subseteq A$ \emph{the saturation of} $A'$.

\begin{corollary}  \label{corostab}
The point $[q:\HH \twoheadrightarrow \FF]\in \Omega^G(\HH,h)$ is GIT-(semi)stable if
and only if inequality \eqref{ek} holds for all \emph{saturated} graded subspaces of $A$.
\end{corollary}

\begin{proof}
The "only if" part is given by Proposition \ref{HMcalculation}. For the "if" part, the proof is analogous to that of \cite[Theorem 3.5]{BT15}.
Let $A' \subseteq A$ be a graded subspace, let $\widetilde{A'}$ be the saturation of $A'$, and let $\FF'$ be the subsheaf of $\FF$ generated by $A'$. 
Assuming that the inequality \eqref{ek} holds for $\widetilde{A'}$, we want to prove that it also holds for $A'$. If $\widetilde{A'}= A'$, then we are done. 
Otherwise, $A' \subsetneq \widetilde{A'}$ and we have 
\begin{align*}
\chi(\widetilde A') - \chi(A')  &= \sum_{\rho \in D_-} \chi_{\rho} \cdot \dim(\widetilde A'/A')_{\rho}\\
&< \sum_{\rho \in D_-} \frac{\kappa(h)}{\dim A} \cdot \dim(\widetilde A'/A')_{\rho}\; , \text{ \ \ by definition of the $\chi_\rho$,}\\
&= \frac{\kappa(h)\cdot \dim(\widetilde A'/A')}{\dim A} = \kappa(h) \cdot \frac{\dim\widetilde A'-\dim A'}{\dim A}.
\end{align*}
It follows that
$$ \dim A \cdot (\kappa_D(\FF') + \chi(A')) - \dim A' \cdot \kappa(h) >  \dim A \cdot (\kappa_D(\FF') + \chi(\widetilde A')) - \dim \widetilde A' \cdot \kappa(h) \geq 0,$$
where the right inequality holds by assumption. Therefore, the inequality \eqref{ek} holds for all graded subspaces $A' \subseteq A$.
\end{proof}

\begin{definition} \label{muDstability}
Let $D \subset \Irr G$ be a finite subset satisfying Hypothesis \ref{hypoD},
and let $\FF$ be an $(\OO_X,G)$-module generated in
$D_-$. We say that $\FF$ is $\mu_D$\emph{-(semi)stable} if, for
all $(\OO_X,G)$-submodules $0 \neq \FF'
\subsetneq \FF$ generated in $D_-$, we have
$$\mu_D(\FF') \sleq \mu_D(\FF)\; ,$$
where
$$\mu_D(\FF'):=\frac{-\kappa_D(\FF')-\chi(\FF')}{r(\FF')}+\frac{\kappa(h)}{r(h)} \ \ \text{($\kappa_D$ and $\chi$ are defined in \S \ref{gamma2})}$$
is the \emph{$D$-slope} of $\FF'$ (and similarly for $\FF$). If $\FF$ is not $\mu_{D}$-semistable, we say that it is \emph{$\mu_{D}$-unstable}. 
\end{definition}

\begin{remark}
Note that the term $\frac{\kappa(h)}{r(h)}$ in the definition of $\mu_D$ is constant and does not affect the comparison 
between slopes in Definition \ref{muDstability}. The reason to place it there is to simplify the comparison in 
\S \ref{comparisonstab} between $\mut$ and $\mu_D$-stability.
\end{remark}

One easily checks that if $\FF$ is a $(G,h)$-constellation, then $\mu_D(\FF)=0$, independently of $D$. 
We saw at the end of \S \ref{gamma2} that we can talk about GIT-(semi)stable $(G,h)$-constellations; the next result makes the connection between GIT-stability and $\mu_D$-stability.  

\begin{corollary} \label{muDisGIT}
Let $\FF$ be a $(G,h)$-constellation, and let $\mu_D$ be the stability condition of Definition \ref{muDstability}.
Then $\FF$ is GIT-(semi)stable if and only if $\FF$ is $\mu_D$-(semi)stable. 
\end{corollary}

\begin{proof}
We have seen a little bit earlier that there is a correspondence between saturated graded subspaces $0 \subsetneq \widetilde{A'} \subsetneq A$ and 
$(\OO_X,G)$-submodules $0 \subsetneq \FF' \subsetneq \FF$ generated in $D_-$. 
In particular, we have $\dim \widetilde{A'}=r(\FF')$, $\chi(\widetilde{A'})=\chi(\FF')$, and $\kappa_D(\widetilde{A'})=\kappa_D(\FF')$. 
Then the result follows from Corollary \ref{corostab}.
\end{proof}

We can finally state the main result of this subsection:

\begin{theorem} \label{HNDexists}
Let $D \subset \Irr G$ be a finite subset satisfying Hypothesis \ref{hypoD}, and let $\mu_D$ be as in Definition \ref{muDstability}. 
Let $\FF$ be a $(G,h)$-constellation. Then $\FF$ has a unique filtration
$$0\subsetneq \GG^D_{1}\subsetneq \GG^D_{2}\subsetneq \cdots \subsetneq \GG^D_{p_D} \subsetneq \GG^D_{p_D+1}=\FF$$
verifying
\begin{enumerate}[(i)]
 \item each $\GG^D_i$ is an $(\OO_X,G)$-submodule generated in $D_-$;
 \item each quotient $\GG^{D,i}:=\GG^D_{i} / \GG^D_{i-1}$ is $\mu_D$-semistable; and
 \item the slopes of the quotients are strictly decreasing
 $$\mu_D(\GG^{D,1})>\mu_D(\GG^{D,2})>\cdots >\mu_D(\GG^{D,p_D})>\mu_D(\GG^{D,p_D+1}).$$
\end{enumerate}
We call this filtration the \emph{$\mu_D$-Harder-Narasimhan filtration} ($\mu_D$-HN filtration for short) of $\FF$.
\end{theorem}

\begin{proof}
The proof is analogous to the one of Theorem \ref{HNexists}.
First we note that the functions $\kappa_{D}$, $\chi$ and $r$ are additive on exact sequences (because so are the Hilbert functions). 
Hence, for any $(\OO_X,G)$-submodule $0 \neq \FF' \subsetneq \FF$ with a non-zero negative part, we can write
$$\mu_D(\FF')=\frac{-\kappa_D(\FF')-\chi(\FF')}{r(\FF')}+\frac{\kappa(h)}{r(h)}=\frac{-\kappa_D(\FF')-\chi(\FF')+r(\FF')\frac{\kappa(h)}{r(h)}}{r(\FF')}\;$$
as the quotient of two additive functions. Then it is clear that $\mu_D$ enjoys the \emph{seesaw property} (see Lemma \ref{ineq1}).
Using the seesaw property and Proposition \ref{finiteness}, we show the existence of a unique maximal $\mu_{D}$-destabilizing subobject $\FF'\subseteq \FF$ (see Proposition \ref{unique}). From this, we easily obtain the existence of the $\mu_D$-HN filtration by induction on $r(\FF)$. 

For the uniqueness part of the statement, Proposition \ref{unik} gives (after replacing the $\mut$-stability by the $\mu_D$-stability) that 
necessarily $\GG^{D}_{1}=\FF'$, and we conclude by induction on the length of the filtration.  
\end{proof}

Let us note that the $\mu_D$-HN filtration of a $(G,h)$-constellation $\FF$ is trivial if and only if $\FF$ is $\mu_D$-semistable 
(equivalently, $\FF$ is GIT-semistable). Explicit examples of $(G,h)$-constellations with non trivial $\mu_D$-HN filtration will be computed in \S \ref{examples}.

\section{Comparison between the different stability notions} \label{comparisonstab}

In this section we compare the $\theta$ or $\mut$-stability, introduced in \S \ref{first}, with the GIT or $\mu_D$-stability, introduced in \S \ref{HN2}. 
In \S \ref{relations}, we summarize the implications between these different stability notions for arbitrary $(\OO_X,G)$-modules generated in $D_-$, and answer several questions remaining open in \cite{BT15}. Then, in \S\S \ref{relations2}--\ref{relations3}, we compare the $\mut$-HN and $\mu_D$-HN filtrations. More precisely, in \S \ref{relations2} we prove that when $D \subset \Irr G$ is a finite subset big enough, the  $\mut$-HN filtration is always a subfiltration of the $\mu_D$-HN filtration (see Theorem \ref{HNterms}). Next, in \S \ref{relations3}, we see how to attach to any $(G,h)$-constellation two convex polygons, 
the $\theta$-polygon and the $D$-polygon, and we prove that the sequence of $D$-polygons converge to the $\theta$-polygon when $D$ grows (Theorem \ref{convergence}).

Throughout this section, $\theta$ is a stability function as in Definition \ref{deftheta}, 
and $D \subset\Irr G$ is a finite subset satisfying Hypothesis \ref{hypoD}.

\subsection{Relations between GIT-stability and \texorpdfstring{$\theta$}{theta}-stability} \label{relations}
We recall that the notions of $\theta$-stability, $\mut$-stability, and $\mu_D$-stability were defined for arbitrary $(\OO_X,G)$-modules generated in $D_-$ in \S\S \ref{first} and \ref{HN2}. 

For all $D$ \emph{big enough} (i.e., $D \subset \Irr G$ is a finite subset which contains a given $\tilde D$ and satisfies Hypothesis \ref{hypoD}), and for all $(\OO_X,G)$-modules generated in $D_-$ and contained in some $(G,h)$-constellation, we have the following implications: 
\begin{equation} \label{diag_relations}
\xymatrix{
   \theta\text{-stable \ } \ar@{=>}[r]^{(a)\ } \ar@{=>}[rd] &\ \mut\text{-stable \ } \ar@{=>}[r] \ar@{=>}[d]^{(b)}  &\ \mut\text{-semistable \ } & \  \theta\text{-semistable} \ar@{=>}[l]_{\ (d)} \\
            &\mu_D\text{-stable \ } \ar@{=>}[r]                 & \ \mu_D\text{-semistable}   \ar@{=>}[u]^{(c)}  & 
  }
\end{equation}
where $(a)$ and $(d)$ follow from Lemma \ref{comparison}, $(b)$ and $(c)$ follow easily from the forthcoming Proposition \ref{wideDD}, and the three other implications are straightforward. Also, for any $(G,h)$-constellation, $(a)$ and $(d)$ are equivalences by Lemma \ref{comparison}, and we have a notion of GIT-stability which coincide with the $\mu_D$-stability by Corollary \ref{muDisGIT}. In particular, for $(G,h)$-constellations, we have
$$
\theta\text{-unstable} \Longleftrightarrow \mut\text{-unstable}  \Longrightarrow  \mu_D\text{-unstable} \Longleftrightarrow  \text{GIT-unstable}.
$$    

An important question, raised in \cite{BT15}, is the following:
\begin{question}   \emph{\cite[Question 5.2]{BT15}}
Are the implications $(b)$ and $(c)$ in Diagram \eqref{diag_relations} equivalences for $(G,h)$-constellations? 
\end{question}
The answer is no in general. Indeed, there are examples of $(G,h)$-constellations 
which are $\mu_D$-stable for all $D$ big enough but never $\mut$-stable, and there are examples of $(G,h)$-constellations 
which are $\mut$-semistable but for which one can find $D$ arbitrarily big such that they are $\mu_D$-unstable. We will compute explicitly such examples in \S \ref{examples}.

Let us mention that, since the answer to \cite[Question 5.2]{BT15} is negative, it is clear that the answer 
to \cite[Question 5.1]{BT15}, which  is about the representability of a certain moduli functor, is also negative. 
However, we do not wish to pursue in this direction in this article.

Once we know that $\mut$-stability and $\mu_D$-stability do not coincide, the next step is to "measure" the difference between these two stability notions. This will be performed in the next subsections.

\subsection{Relations between the filtrations}  \label{relations2}
In this section we prove one of the main results of this paper. We show that the $\mut$-HN filtration is always a subfiltration of the $\mu_D$-HN filtration
when $D \subset \Irr G$ is a finite subset big enough (Theorem \ref{HNterms}).

First, we need two preliminary results. The first one assures that $(\OO_X,G)$-modules 
generated in $D_-$ are both Noetherian and Artinian. The second one assures the convergence $\mu_D(.) \to \mut(.)$, when $D$ tends 
to $\supp h$, for subsheaves and quotients of $(G,h)$-constellations.

\begin{lemma}  \label{noeth}
Let $\FF$ be an $(\OO_X,G)$-module generated in $D_{-}$. Every increasing chain
$$ \FF_{1}\subsetneq \FF_{2}\subsetneq \FF_{3}\subsetneq \cdots \subsetneq \FF $$
and every decreasing chain
$$\cdots \subsetneq \FF_{3}\subsetneq \FF_{2}\subsetneq \FF_{1}\subsetneq \FF$$
of $(\OO_X,G)$-submodules generated in $D_-$ has length at most $r(\FF)$.
\end{lemma}

\begin{proof}
This result is a direct consequence of the definition of $r(\FF)$ in \S \ref{first}, using the fact that two $(\OO_X,G)$-submodules $\FF' \subseteq \FF'' \subseteq \FF$, both generated in $D_-$, coincide if and only if $r(\FF')=r(\FF'')$. 
\end{proof}

\begin{proposition}  \label{wideDD}
We fix $\epsilon >0$. There exists a finite subset $D_\epsilon \subset \Irr G$ satisfying Hypothesis \ref{hypoD} such that 
for all $(G,h)$-constellations $\FF$ and all $(\OO_X,G)$-submodules $0 \subseteq \FF' \subsetneq \FF'' \subseteq \FF$ generated in $D_-$, we have $$|\mu_\theta(\FF''/\FF')-\mu_D(\FF''/\FF')|<\epsilon$$ 
for all $D \supset D_\epsilon$.
\end{proposition}

\begin{proof}
We fix $D \subset \Irr G$ a finite subset satisfying Hypothesis \ref{hypoD}.
Let $h'$ be an arbitrary Hilbert function whose support intersects $D_-$ and such that $h' \leq h$, where $\leq$ denotes the partial order 
defined by \eqref{order}. Note that all Hilbert functions corresponding to quotients $\FF''/\FF'$ of subsheaves $0 \subseteq \FF' \subsetneq \FF'' \subseteq \FF$ are of this kind.  

From Definition \ref{muDstability} and the choice of GIT parameters \S \ref{gamma2}, a direct calculation gives
$$ 
\mu_{D}(h')=\frac{-\sum_{\rho\in D}\theta_{\rho}h'(\rho)-\frac{S_{D}}{d}\sum_{\rho\in D\backslash D_{-}}\frac{h'(\rho)}{h(\rho)}}{r(h')}\;.
$$
We deduced from this that the difference between $\mut$ and $\mu_D$ is  given by
$$\mu_\theta(h')-\mu_D(h')=\frac{-\sum_{\rho\notin D}\theta_{\rho}h'(\rho)+\frac{S_D}{d} \sum_{\rho \in D \backslash D_-}{\frac{h'(\rho)}{h(\rho)}}}{r(h')} \; .$$
Since $h' \leq h$, we have $\sum_{\rho \in D \backslash D_-}{\frac{h'(\rho)}{h(\rho)}} \leq d$, and 
$$0 \leq \sum_{\rho\notin D}\theta_{\rho}h'(\rho)\leq \sum_{\rho\notin D}\theta_{\rho}h(\rho)=S_{D}.$$
Now, it is clear from the definition of $S_D$ that for every $\epsilon>0$, 
we can find $D_\epsilon$ satisfying Hypothesis \ref{hypoD} such that if $D \supset D_\epsilon$, then $S_D<\frac{\epsilon}{2}$.
Thus, 
$$|\mu_\theta(h')-\mu_D(h')| \leq \frac{2S_{D}}{r(h')}<\epsilon\; .$$
\end{proof}

Let $\FF$ be a $(G,h)$-constellation, and let $0 \subseteq \FF' \subsetneq \FF'' \subseteq \FF$ be $(\OO_X,G)$-submodules
generated in $D_-$. We denote by $\mathcal{A}=\{\alpha_1,\alpha_2,\ldots,\alpha_s\}$ the set of pairwise different 
possible values for $\mut(\FF''/\FF')$. The fact that $\mathcal{A}$ is finite follows from Proposition \ref{finiteness}.
Then, we define
\begin{equation} \label{epsilon}
\epsilon_0:=\frac{1}{4} \min_{i \neq j} |\alpha_i-\alpha_j|.
\end{equation}

The next Theorem says that, given a finite subset $D\subset \Irr G$ big enough, all terms of the $\mut$-HN filtration already 
appear in the $\mu_D$-HN filtration, although the latter can contain more terms in general.

\begin{theorem}
\label{HNterms}
Let $D\subset \Irr G$ be a finite subset satisfying Hypothesis \ref{hypoD} and containing $D_{\epsilon_0}$,
where $\epsilon_0$ is defined by \eqref{epsilon} and $D_{\epsilon_0}$ is given by Proposition \ref{wideDD}.
Let $\FF$ be a $(G,h)$-constellation.
We denote the $\mu_\theta$-HN filtration of $\FF$ by
$$ 0 \subsetneq \FF_1 \subsetneq\FF_2 \subsetneq \FF_3 \subsetneq \cdots \subsetneq\FF_t\subsetneq \FF_{t+1}=\FF,$$
and the $\mu_D$-HN filtration of $\FF$ (where we drop the index $D$ to simplify the notation) by
$$ 0 \subsetneq \GG_1 \subsetneq \GG_2 \subsetneq \GG_3 \subsetneq \cdots \subsetneq \GG_p\subsetneq \GG_{p+1}=\FF.$$
Then, $p \geq t$ and there exists a subset of indices $1 \leq i_1<i_2<\ldots<i_t < i_{t+1}=p+1$ such that $\FF_j=\GG_{i_j}$ for all $1 \leq j \leq t+1$.
Moreover, for all $j=1,\ldots,t+1$ and all $k=i_{j-1}+1\ldots,i_j$, we have $\mut(\FF_j/\FF_{j-1})=\mut(\GG_{k}/\GG_{i_{j-1}})$.
\end{theorem}

\begin{proof}
The proof goes by induction on $t$.
If $t=0$, then the first part of the result is obvious, the second part follows from Proposition \ref{wideDD}.
We now suppose that $t \geq 1$.
We want to prove that there exists $i_1 \geq 1$ such that $\GG_{i_1}=\FF_1$.
If $\GG_1=\FF_1$, then we are done.
Otherwise let us prove that $\GG_1 \subsetneq \FF_1$ and that $\FF_1/\GG_1$ is a $\mu_D$-destabilizing subobject of $\FF/\GG_1$.
If $\mut(\FF_1)>\mut(\GG_1)$, then $\mu_D(\FF_1)>\mu_D(\GG_1)$ by definition of $\epsilon_{0}$ and by Proposition \ref{wideDD}, which contradicts the definition of $\GG_1$. 
Hence $\mut(\FF_1)=\mut(\GG_1)$, and thus $\GG_1 \subsetneq \FF_1$ by Proposition \ref{unik}. 
The seesaw property (Lemma \ref{ineq1}) gives $\mut(\FF_1)=\mut(\FF_1/\GG_1)$ and $\mut(\FF/\GG_1)<\mut(\FF)$. 
We deduce that 
$$\mut(\FF/\GG_1)<\mut(\FF)<\mut(\FF_1/\GG_1)$$ 
and $\FF_1/\GG_1$ is a $\mut$-destabilizing subobject of $\FF/\GG_1$.
Again, by using the definition of $\epsilon_{0}$ and Proposition \ref{wideDD}, we see that $\FF_1/\GG_1$ is a $\mu_{D}$-destabilizing subobject of $\FF/\GG_1$.

Therefore $p\geq 2$ and $\mu_D(\GG_2/\GG_1) \geq \mu_D(\FF_1/\GG_1)$ by definition of $\GG_2$.
If $\GG_2=\FF_1$, then we are done.
Otherwise, let us prove that $\GG_2 \subsetneq \FF_1$ and that $\FF_1/\GG_2$ is a $\mu_D$-destabilizing subobject of $\FF/\GG_2$.
We consider the exact sequence 
$$0 \to \GG_1 \to \GG_2 \to \GG_2/\GG_1 \to 0.$$
Since $\mu_\theta(\GG_1)=\mu_\theta(\FF_1)$ is maximal, the seesaw property gives 
$$\mu_\theta(\GG_1) \geq  \mu_\theta(\GG_2) \geq \mu_\theta(\GG_2/\GG_1).$$
If those inequalities were strict, using $\mut(\GG_1)=\mut(\FF_1/\GG_1)$, the definition of $\epsilon_{0}$ and 
Proposition \ref{wideDD}, we would have
$$\mu_{D}(\FF_1/\GG_1)>\mu_D(\GG_2)>\mu_D(\GG_2/\GG_1) \geq \mu_D(\FF_1/\GG_1),$$
where the last inequality is by definition of $\GG_2$, which is a contradiction. 
Hence, we necessarily have 
$$\mut(\FF_1)=\mut(\GG_1) =  \mut(\GG_2) = \mut(\GG_2/\GG_1).$$ 
In particular, $\GG_2 \subsetneq \FF_1$ by Proposition \ref{unik}.
Also, the seesaw property implies that
$$\mut(\FF/\GG_2)<\mut(\FF)<\mut(\FF_1)=\mut(\FF_1/\GG_2),$$ 
and thus $\FF_1/\GG_2$ is a $\mut$-, therefore $\mu_D$- by Proposition \ref{wideDD},
destabilizing subobject of $\FF/\GG_2$.

Therefore $p\geq 3$ and $\mu_D(\GG_3/\GG_2) \geq \mu_D(\FF_1/\GG_2)$ by definition of $\GG_3$.
If $\GG_3=\FF_1$, then we are done. Otherwise we follow the same argument to construct subobjects $\GG_4$, $\GG_5$, etc, such that
$$0 \subsetneq \GG_1 \subsetneq \GG_2 \subsetneq \GG_3 \subsetneq \GG_4 \subsetneq \GG_5 \subsetneq\cdots \subsetneq \FF_1\; .$$
By Lemma \ref{noeth}, every increasing sequence has to stabilize, and thus there exists $1 \leq i_1 \leq p$ such that $\GG_{i_1}=\FF_1$. 
Moreover, it is clear that for all $1 \leq k \leq i_1$, we have $\mut(\GG_k)=\mut(\FF_1)$.

Now the $\mut$-HN filtration of $\FF/\FF_1$ is given by
$$ 0 \subsetneq \FF_2/\FF_1 \subsetneq \FF_3/\FF_1 \subsetneq \cdots \subsetneq \FF_{t}/\FF_1\subsetneq \FF_{t+1}/\FF_1=\FF/\FF_1,$$
and the $\mu_D$-HN filtration of $\FF/\FF_1$ is given by
$$ 0 \subsetneq \GG_{i_{1}+1}/\FF_1 \subsetneq \GG_{i_{1}+2}/\FF_1 \subsetneq \cdots \subsetneq \GG_{p}/\FF_{1} \subsetneq \GG_{p+1}/\FF_1=\FF/\FF_1.$$
Since the length of the $\mut$-HN filtration of $\FF/\FF_1$ is one less than the length of the $\mut$-HN filtration of $\FF$, we conclude by induction.
\end{proof}

We will see in \S \ref{examples} that  $\mut$-HN and $\mu_D$-HN filtrations need not coincide. Actually, the $\mu_D$-HN 
filtration need not stabilize when $D$ tends to $\supp h$, even the number of terms of the $\mu_D$-HN filtration can oscillate when $D$ grows.

\begin{remark} \label{HNtoJH}
With the notation of Theorem \ref{HNterms}, we saw that $\mut(\FF_{j}/\FF_{j-1})=\mut(\GG_{k}/\GG_{i_{j-1}})$ for all $1 \leq j \leq t+1$ and all $i_{j-1}+1 \leq k \leq i_j$. 
It follows that 
$$\GG_{i_{j-1}+1}/ \GG_{i_{j-1}}  \subsetneq  \GG_{i_{j-1}+2}/ \GG_{i_{j-1}} \subsetneq \cdots \subsetneq \GG_{i_j}/ \GG_{i_{j-1}}=\FF_j/\FF_{j-1}$$
is a subfiltration of a (generally non-unique) $\mut$-Jordan-H\"older filtration (see Remark \ref{JH}) of the $\mut$-semistable factor $\FF_j/\FF_{j-1}$.
\end{remark}

\subsection{Relations between the polygons} \label{relations3}
In this subsection, we explain how to associate a $\theta$-polygon resp. a $D$-polygon, to any $(G,h)$-constellation $\FF$. 
Those polygons usually appear in literature (see for instance \cite{Sha:1977}) as a convenient way to encode numerical information regarding the $\mut$-HN and the $\mu_D$-HN filtrations of $\FF$. 
For instance, they encode the length of the filtration as well as the slope of each subsheaf. 
The "only" piece of information that we lose when considering those polygons instead of the actual filtrations is the explicit generators of each subsheaf.

Then we will prove that, even though the $\mut$-HN and $\mu_D$-HN filtrations do not coincide in general, we have uniform convergence 
of the sequence of $D$-polygons to the $\theta$-polygon when $D$ grows (Theorem \ref{convergence}). Explicit examples of such polygons will be 
computed in \S \ref{examples}.

\begin{definition} \label{defpolygons}
Let $\FF$ be a $(G,h)$-constellation, and let $\FF_\bullet$ and $\GG_{\bullet}$ be the $\mut$-HN and $\mu_D$-HN filtrations of $\FF$ respectively. 
We call \emph{$\theta$-polygon} of $\FF$ to the convex hull of the points with coordinates
$$(r(\FF_{i}),w^{\theta}_{i}),\ \text{ where } \ w^{\theta}_{i}=r(\FF_{i})\cdot \mut(\FF_{i})\; .$$
Similarly, we call \emph{$D$-polygon} to the convex hull of the points with coordinates 
$$(r(\GG_{i}),w^{D}_{i}),\ \text{ where } \ w^{D}_{i}=r(\GG_{i})\cdot \mu_{D}(\GG_{i})\; .$$
\end{definition}

\begin{figure}[h]
\begin{center}
$$\begin{tikzpicture}
\draw [->] (-5,0) -- (5,0);
\draw [->] (-5,0) -- (-5,4);
\node [left] at  (-5,3.5) {$w^{D}_{i}, w^{\theta}_{i}$};
\node [below] at (4.5,0) {$r(.)$};
\draw [-,ultra thick] (-5,0) -- (-3,2);
\draw [-,ultra thick] (-3,2) -- (0,3);
\draw [-, ultra thick] (0,3) -- (3,2);
\draw [dashed, ultra thick] (3,2) -- (4,1);
\draw [-, thick] (-5,0) -- (-4,1.5);
\draw [-, thick] (-4,1.5) -- (-3,2.5);
\draw [-, thick] (-3,2.5) -- (-1.5,3.2);
\draw [-, thick] (-1.5,3.2) -- (0,3.5);
\draw [-, thick] (0,3.5) -- (1,3.4);
\draw [-, thick] (1,3.4) -- (2,3.1);
\draw [-, thick] (2,3.1) -- (3,2.4);
\draw [dashed, thick] (3,2.4) -- (3.5, 1.8);
\draw [dotted] (-3,2) -- (-3,2.5);
\draw [dotted] (0,3) -- (0,3.5);
\draw [dotted] (3,2) -- (3,2.4);
\draw [<->, thick] (2,2.3) -- (2,3.1);
\node[below] at (0,2.9) {$\FF_{j-1}$};
\node [below] at (3,1.9) {$\FF_{j}$};
\node [above] at (0,3.5) {$\GG_{i_{j-1}}$};
\node [above] at (2.1,3.1) {$\GG_{i_{j-1}+k}$};
\node [above] at (3.1,2.4) {$\GG_{i_{j}}$};

\node [below right] at (-3,2) {$\theta$\text{-polygon}};
\node [above left] at (-1.5,3.1) {$D$\text{-polygon}};
\end{tikzpicture}$$
\caption{$\theta$-polygon and $D$-polygon of a $(G,h)$-constellation}
\label{convpolygons}
\end{center}
\end{figure}
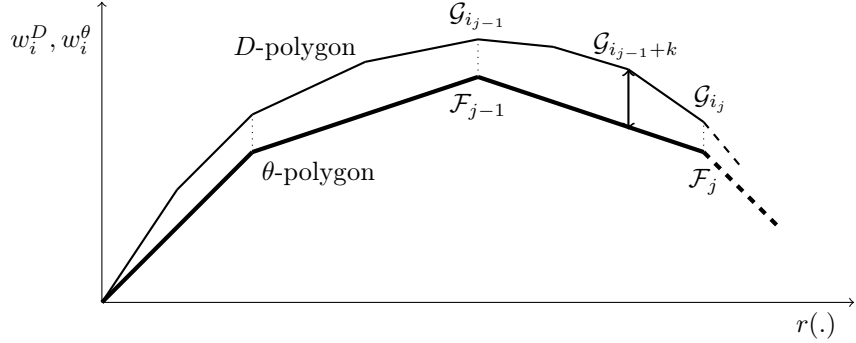

Figure \ref{convpolygons} shows the $\theta$-polygon and the $D$-polygon of filtrations $\mathcal{F}_{\bullet}$ and $\mathcal{G}_{\bullet}$ respectively ; it also makes more visual the conclusions of Theorem \ref{HNterms}.
By construction, the $\theta$-polygon and the $D$-polygons identify with the graph of concave piecewise linear functions 
$$f_{\theta},f_{D}: [0\,,\,r(h)] \to \RR_{\geq 0},\;\; \text{where  }f_{\theta}(r(\FF_i))=w_{i}^{\theta}\; \text{  and  } \;f_{D}(r(\GG_i))=w_{i}^{D}\; ,$$ 
satisfying $f(0)=f(r(h))=0$. We can thus talk about convergence of a sequence of polygons.

\begin{definition}
Let $f_\theta$ resp. $f_D$, be the function whose graph is the $\theta$-polygon resp. the $D$-polygon. 
Then we say that the sequence of $D$-polygons \emph{converges uniformly} to the $\theta$-polygon if
$$\forall \epsilon>0,\; \exists \widetilde D \subset \Irr G,\; \forall D \supset \widetilde D \text{\ (satisfying Hypothesis \ref{hypoD})},\;\; |\!| f_D-f_\theta|\!|_{\infty}<\epsilon. $$
\end{definition}

\begin{theorem} \label{convergence}
Let $\FF$ be a $(G,h)$-constellation. The sequence of $D$-polygons of $\FF$ converge uniformly to the $\theta$-polygon of $\FF$ when the finite subset $D \subset \Irr G$ (which satisfies Hypothesis \ref{hypoD}, and thus is contained in $\supp h$) converges to $\supp h$.
\end{theorem}

\begin{proof}
We fix $\epsilon >0$, and we take $D$ containing $D_{\epsilon_0}$ so that Theorem \ref{HNterms} holds.  
Let
$$ 0 \subsetneq \FF_1 \subsetneq \FF_2 \subsetneq \FF_3 \subsetneq \cdots \subsetneq \FF_{t} \subsetneq \FF_{t+1}=\FF$$
be the $\mu_\theta$-HN filtration of $\FF$, and let
$$ 0 \subsetneq \GG_1 \subsetneq \GG_2 \subsetneq \cdots \subsetneq \GG_{i_{j-1}}\subsetneq \cdots \subsetneq \GG_{i_{j-1}+k}\subsetneq \cdots \subsetneq \GG_{i_{j}}\subsetneq \cdots \subsetneq \GG_{p}\subsetneq \GG_{p+1}=\FF$$ 
be the $\mu_D$-HN filtration of $\FF$, where $1 \leq i_1<i_2<\ldots<i_t <i_{t+1}= p+1$ are the indexes such that $\FF_j=\GG_{i_j}$ for all $1 \leq j \leq t+1$.

The functions $f_\theta$ and $f_D$, associated with the $\theta$-polygon and the $D$-polygons respectively, are piecewise linear. 
Also, the set of abcissae of vertices of the $\theta$-polygon is always contained in the set of abcissae of vertices of the $D$-polygon. 
Consequently, to prove that the $D$-polygons converge to the $\theta$-polygon, it suffices to bound
$$ \max_{l=1,\ldots,p} |f_D(r(\GG_l))-f_\theta(r(\GG_l))|.$$

Let $\GG_{i_{j-1}+k}$ be a term of the $\mu_D$-HN filtration of $\FF$. Then
\begin{align*}
&|f_D(r(\GG_{i_{j-1}+k}))-f_\theta(r(\GG_{i_{j-1}+k}))|\\
&=|w_{i_{j-1}+k}^{D}-\left(  w_{j-1}^{\theta}+\left(r(\GG_{i_{j-1}+k})-r(\GG_{i_{j-1}})\right)\mut(\FF_{j}/\FF_{j-1}) \right)| \\
&=|r(\GG_{i_{j-1}})\mu_D(\GG_{i_{j-1}})+\left(r(\GG_{i_{j-1}+k})-r(\GG_{i_{j-1}})\right)\mu_{D}(\GG_{i_{j-1}+k}/\GG_{i_{j-1}}) \\
 &\ -r(\FF_{j-1})\mut(\FF_{j-1})-\left(r(\GG_{i_{j-1}+k})-r(\GG_{i_{j-1}})\right)\mut(\FF_{j}/\FF_{j-1})|\\
&\leq(r(\GG_{i_{j-1}+k})-r(\GG_{i_{j-1}}))|\mu_{D}(\GG_{i_{j-1}+k}/\GG_{i_{j-1}})- \mut(\FF_{j}/\FF_{j-1})|\\
&\ +r(\FF_{j-1})|\mu_{D}(\GG_{i_{j-1}})-\mut(\FF_{j-1})|\;.
\end{align*}
By Theorem \ref{HNterms}, we have $\mut(\GG_{i_{j-1}+k}/\GG_{i_{j-1}})=\mut(\FF_{j}/\FF_{j-1})$. We denote $\epsilon':=\frac{\epsilon}{r(h)}$, 
and we suppose that $D$ contains $D_{\epsilon'}$. Then Proposition \ref{wideDD} implies that
$$|f_D(r(\GG_{i_{j-1}+k}))-f_\theta(r(\GG_{i_{j-1}+k}))| <(r(\GG_{i_{j-1}+k})-r(\GG_{i_{j-1}})) \epsilon' +  r(\FF_{j-1}) \epsilon'<\epsilon.$$
So denoting $\widetilde D:= D_{\epsilon_0} \cup D_{\epsilon'}$, we see that for all $D\supset \widetilde D$, we have 
$$|\!| f_D-f_\theta|\!|_{\infty}=\max_{l=1,\ldots,p} |f_D(r(\GG_l))-f_\theta(r(\GG_l))|<\epsilon.$$
\end{proof}

\section{Examples} \label{examples}
In this last section we present several examples to illustrate the different phenomena that we considered throughout this article. By $D$ we always mean a finite subset of $\Irr G$ satisfying Hypothesis \ref{hypoD}.

Since all these phenomena are already visible in small dimension, we will stick to the following quite simple framework. 
Let $G=\Gm$ be the multiplicative group. We recall that $\ZZ$ identifies with $\Irr G$, the set of isomorphy classes of irreducible representations of $G$, via the map $r \in \ZZ \mapsto V_r \in \Irr G$, where $V_r$ is the $1$-dimensional representation on which $t \in G$ acts by multiplication by $t^r$. 
We consider the action of $G$ on the algebra $\CC[x,y]$ defined by $t.x:=tx$ and $t.y:=t^{-1}y$, for all $t \in G$.
With this action, note that the weight of a monomial $x^{a}y^{b}$ is $a-b$. 
We take $X:=\Spec \CC[x,y]/(xy)$. Then we have
$$\CC[x,y]/(xy) \cong \CC[x]_{>0} \oplus \CC \oplus \CC[y]_{>0} \cong \bigoplus_{r \in \ZZ} V_r$$  
as $G$-modules. Let $h: \ZZ \to \NN$ be the Hilbert function defined by $h(r)=1$, for all $r \in \ZZ$. 
Then it is clear that $\OO_X$ is a $(G,h)$-constellation on $X$, provided that we choose $\theta$ such that $D_-$ contains $\{0\}$ (to ensure that $\OO_X$ is generated in $D_-$). 

We now consider the $(\OO_X,G)$-submodules of $\OO_X$. Those correspond to the $G$-stable ideals of $\OO_X$ and are of three kinds. 
\begin{enumerate}[(i)]
\item $I_p:=(\overline{x}^p)$, with $p \geq 1$, then $h_{I_p}(r)=1$ for $r \geq p$ and $h_{I_p}(r)=0$ otherwise.
\item $J_q:=(\overline{y}^q)$, with $q \geq 1$, then $h_{J_q}(r)=1$ for $r \leq -q$ and $h_{J_q}(r)=0$ otherwise. 
\item $K_{p,q}:=(\overline{x}^p,\overline{y}^q)$ with $p,q \geq 1$, then $h_{K_{p,q}}(r)=1$ for $(r \geq p \text{ or } r \leq -q)$ and $h_{K_{p,q}}(r)=0$ otherwise. 
\end{enumerate}
Here we denote by $\overline{x}$ and $\overline{y}$ the images of $x$ and $y$ in $\CC[x,y]/(xy)$ respectively.  
Geometrically, $X$ is simply the union of the two coordinate axes in the plane $\mathbb{A}_\CC^2$, $I_p$ is the ideal of the vertical thick line $(\overline{x}^p=0)$, $J_q$ is the ideal of the horizontal thick line $(\overline{y}^q=0)$, and $K_{p,q}$ is the ideal of the thick point $(\overline{x}^p=0=\overline{y}^q)$. 

First, we begin with two examples to show that implications $(b)$ and $(c)$ of Diagram \eqref{diag_relations} are not equivalences in general. 
In particular, this answers negatively \cite[Question 5.2]{BT15}. To compute $\mu_D(h')$ in our forthcoming examples, we will use the formula
\begin{equation} \label{form}
\mu_D(h')=\frac{-1}{r(h')} \left( \sum_{\rho \in D} \theta_\rho h'(\rho)+ \frac{S_D}{d} \sum_{\rho \in D \setminus D_-} \frac{h'(\rho)}{h(\rho)} \right),
\end{equation}
where $S_D$ and $d$ are defined in \S \ref{gamma2}. This formula is obtained simply by plugging the numerical values given in \S \ref{gamma2} in Definition \ref{muDstability}.

\begin{example}\label{GITstablebutnotstable}
Example of a $(G,h)$-constellation $\mu_D$-stable, for all finite subsets $D \subset \Irr G$ big enough, but strictly $\mut$-semistable.
Let $\theta$ be the stability function defined as follows:
$$\begin{tikzpicture}[scale=0.9]
\draw [-] (-6,0) -- (6.8,0);
\draw (-6,0.1) node[above]{$r=$} ;
\draw (-6,-0.1) node[below]{$\theta_r=$} ;

\draw (-5,0.1) node[above]{$-k$} ;
\draw (-5,0) node  {$\bullet$} ;
\draw (-5,-0.1) node[below]{$\frac{1}{2^{k-2}}$} ;

\draw (-3.8,0.2) node[above]{$\cdots$} ;
\draw (-3.8,-0.3) node[below]{$\ldots$} ;

\draw (-3,0.1) node[above]{$-4$} ;
\draw (-3,0) node  {$\bullet$} ;
\draw (-3,-0.1) node[below]{$\frac{1}{4}$} ;

\draw (-2,0.1) node[above]{$-3$} ;
\draw (-2,0) node  {$\bullet$} ;
\draw (-2,-0.1) node[below]{$\frac{1}{2}$} ;

\draw (-1,0.1) node[above]{$-2$} ;
\draw (-1,0) node  {$\bullet$} ;
\draw (-1,-0.1) node[below]{$0$} ;

\draw (0,0.1) node[above]{$-1$} ;
\draw (0,0) node  {$\bullet$} ;
\draw (0,-0.1) node[below]{$0$} ;

\draw (1,0.1) node[above]{$0$} ;
\draw (1,0) node  {$\bullet$} ;
\draw (1,-0.1) node[below]{$-1$} ;

\draw (2,0.1) node[above]{$1$} ;
\draw (2,0) node  {$\bullet$} ;
\draw (2,-0.1) node[below]{$-1$} ;

\draw (3,0.1) node[above]{$2$} ;
\draw (3,0) node  {$\bullet$} ;
\draw (3,-0.1) node[below]{$\frac{1}{2}$} ;

\draw (4,0.1) node[above]{$3$} ;
\draw (4,0) node  {$\bullet$} ;
\draw (4,-0.1) node[below]{$\frac{1}{4}$} ;

\draw (5,0.1) node[above]{$4$} ;
\draw (5,0) node  {$\bullet$} ;
\draw (5,-0.1) node[below]{$\frac{1}{8}$} ;

\draw (5.7,0.2) node[above]{$\cdots$} ;
\draw (5.7,-0.3) node[below]{$\ldots$} ;

\draw (6.5,0.1) node[above]{$k$} ;
\draw (6.5,0) node  {$\bullet$} ;
\draw (6.5,-0.1) node[below]{$\frac{1}{2^{k-1}}$} ;
\end{tikzpicture}$$

We take $\FF:=\OO_X$. Then we have $\theta(\FF)=\sum_{r \in \ZZ} \theta_r h(r)=\sum_{r \in \ZZ} \theta_r=0$, 
therefore $\theta$ satisfies the conditions of Definition \ref{deftheta}. 
Since $D_-=\{0,1\}$, the only $(\OO_X,G)$-submodule of $\FF$ generated in $D_-$ is $\FF':=I_1$. 
We have 
$$\theta(\FF')=\sum_{r \in \ZZ} \theta_r h_{I_1}(r)= \sum_{r \geq 1} \theta_r=0,$$ 
hence $\FF$ is strictly $\mut$-semistable, i.e. $\FF$ is $\mut$-semistable but not $\mut$-stable. 

On the other hand, let us verify that $\FF$ is $\mu_D$-stable when $D$ is big enough. 
Let $D=D_N:=[-N,N]\subset \ZZ=\Irr G$. There exists $N_0 \geq 3$ such that $D_N$ satisfies Hypothesis \ref{hypoD} for all $N \geq N_0$. 
An explicit computation with \eqref{form} gives
\begin{align*}
\mu_{D_N}(\FF')&=\frac{-1}{r(\FF')} \left( \sum_{r\in D_{N}} \theta_r h_{I_{1}}(r)+ \frac{S_{D_{N}}}{d} 
\sum_{r\in D_{N} \setminus D_-} \frac{h_{I_{1}}(r)}{h(r)} \right)\\
&=(-1)\cdot \left(\frac{-1}{2^{N-1}}+\frac{\frac{3}{2^{N-1}}}{2N-1}(N-1)\right)=\frac{1}{2^{N-1}}\left(\frac{-N+2}{2N-1}\right)<0\; .
\end{align*}
Hence, for all $N \geq N_0$, we have $\mu_{D_N}(\FF')<\mu_{D_N}(\FF)=0$, i.e., $\FF$ is $\mu_{D_N}$-stable. 
Finally, it is easy to check that the same holds for all $D$ big enough.
\end{example}

\begin{example} \label{ex2}
Example of a $(G,h)$-constellation $\FF$ which is $\mut$-semistable but $\mu_{D}$-unstable for some finite subsets $D \subset \Irr G$ arbitrarily big.
Let $\theta$ be the stability function defined as follows:
$$\begin{tikzpicture}
\draw [-] (-6.5,0) -- (5.5,0);
\draw (-6,0.1) node[above]{$r=$} ;
\draw (-6,-0.1) node[below]{$\theta_r=$} ;

\draw (-5,0.1) node[above]{$-k$} ;
\draw (-5,0) node  {$\bullet$} ;
\draw (-5,-0.1) node[below]{$0$} ;

\draw (-4,0.2) node[above]{$\cdots$} ;
\draw (-4,-0.3) node[below]{$\ldots$} ;

\draw (-3,0.1) node[above]{$-3$} ;
\draw (-3,0) node  {$\bullet$} ;
\draw (-3,-0.1) node[below]{$0$} ;

\draw (-2,0.1) node[above]{$-2$} ;
\draw (-2,0) node  {$\bullet$} ;
\draw (-2,-0.1) node[below]{$0$} ;

\draw (-1,0.1) node[above]{$-1$} ;
\draw (-1,0) node  {$\bullet$} ;
\draw (-1,-0.1) node[below]{$1$} ;

\draw (0,0.1) node[above]{$0$} ;
\draw (0,0) node  {$\bullet$} ;
\draw (0,-0.1) node[below]{$-1$} ;

\draw (1,0.1) node[above]{$1$} ;
\draw (1,0) node  {$\bullet$} ;
\draw (1,-0.1) node[below]{$-1$} ;

\draw (2,0.1) node[above]{$2$} ;
\draw (2,0) node  {$\bullet$} ;
\draw (2,-0.1) node[below]{$\frac{1}{2}$} ;

\draw (3,0.1) node[above]{$3$} ;
\draw (3,0) node  {$\bullet$} ;
\draw (3,-0.1) node[below]{$\frac{1}{4}$} ;

\draw (4,0.2) node[above]{$\cdots$} ;
\draw (4,-0.3) node[below]{$\ldots$} ;

\draw (5,0.1) node[above]{$k$} ;
\draw (5,0) node  {$\bullet$} ;
\draw (5,-0.1) node[below]{$\frac{1}{2^{k-1}}$} ;
\end{tikzpicture}$$

As in the previous example, we take $\FF:=\OO_X$. Then we have $\theta(\FF)=0$, therefore $\theta$ satisfies the conditions of Definition \ref{deftheta}. 
The only $(\OO_X,G)$-submodule of $\FF$ generated in $D_-=\{0,1\}$ is $\FF':=I_1$, and we have $\theta(\FF')=0$, i.e., 
$\FF$ is strictly $\mut$-semistable.  

On the other hand, let us verify that $\FF$ is $\mu_D$-unstable for some finite subset $D\subset \Irr G$ big enough. 
We denote again $D=D_N:=[-N,N] \subset \ZZ$. 
There exists $N_0 \geq 1$ such that $D_N$ satisfies Hypothesis \ref{hypoD} for all $N \geq N_0$. 
A similar computation using \eqref{form} gives 
\begin{align*}
\mu_{D_N}(\FF')&=\frac{-1}{r(\FF')} \left( \sum_{r\in D_{N}} \theta_r h_{I_{1}}(r)+ \frac{S_{D_{N}}}{d} 
\sum_{r\in D_{N} \setminus D_-} \frac{h_{I_{1}}(r)}{h(r)} \right)\\
&=(-1)\cdot \left(\frac{-1}{2^{N-1}}+\frac{\frac{1}{2^{N-1}}}{2N-1}(N-1)\right)=\frac{1}{2^{N-1}}\left(\frac{N}{2N-1}\right)>0\; .
\end{align*}
Hence, for all $N \geq N_0$, we have $\mu_{D_N}(\FF')>\mu_{D_N}(\FF)=0$, i.e., $\FF$ is $\mu_{D_N}$-unstable. 

To summarize, in this example the $\mut$-HN filtration of $\FF$ is trivial (since $\FF$ is $\mut$-semistable), 
but the $\mu_{D_N}$-HN filtration of $\FF$ is $0 \subsetneq \FF' \subsetneq \FF$ for all $N \geq N_0$. 
The latter is also the $\mut$-Jordan-H\"older filtration of $\FF$; see Remark \ref{HNtoJH}.  
Figure \ref{pictureex2} illustrates the behavior of the $\theta$-polygon and the $D_N$-polygons for different $N$. 

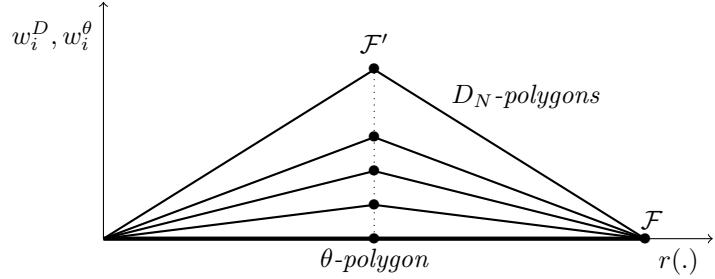
\begin{figure}[h]
\begin{center}
$$\begin{tikzpicture}[scale=0.9]
\draw [->] (-4,0) -- (5,0);
\draw [->] (-4,0) -- (-4,3.5);
\node [left] at  (-4,3) {$w^{D}_{i}, w^{\theta}_{i}$};
\node [below] at (4.5,0) {$r(.)$};
\draw [-, ultra thick] (-4,0) -- (4,0);
\draw (0,0) node  {$\bullet$} ;
\draw (4,0) node  {$\bullet$} ;
\draw [-, thick] (-4,0) -- (0,2.5);
\draw (0,2.5) node  {$\bullet$} ;
\draw [-, thick] (0,2.5) -- (4,0);
\draw [-, thick] (-4,0) -- (0,1.5);
\draw (0,1.5) node  {$\bullet$} ;
\draw [-, thick] (0,1.5) -- (4,0);
\draw [-, thick] (-4,0) -- (0,1);
\draw (0,1) node  {$\bullet$} ;
\draw [, thick] (0,1) -- (4,0);
\draw [-, thick] (-4,0) -- (0,0.5);
\draw (0,0.5) node  {$\bullet$} ;
\draw [, thick] (0,0.5) -- (4,0);
\draw [dotted] (0,0) -- (0,2.5);
\node [above] at (0,2.6) {$\FF'$};
\node [above] at (4.1,0) {$\FF$};

\node [below] at (0,0) {$\theta$\text{-polygon}};
\node [right] at (1,2.1) {$D_{N}$\text{-polygons}};
\end{tikzpicture}$$
\caption{$\theta$-polygon and $D_{N}$-polygons of Example \ref{ex2}}
\label{pictureex2}
\end{center}
\end{figure}
\end{example}

One might ask whether the reason for the $\mut$-HN and the $\mu_D$-HN filtrations to be distinct in Example 
\ref{ex2} is that the $(G,h)$-constellation $\FF$ is $\mut$-semistable (which "forces" the $\mut$-HN filtration to be trivial). 
The answer is actually negative as the next two examples will show.

\begin{example}  \label{ex3}
Example where the $\mut$-HN and $\mu_D$-HN filtrations are both non-trivial, and the number of terms of the $\mu_D$-HN filtration does not stabilize when the finite subset $D \subset \Irr G$ grows.
Let $\theta$ be the stability function defined as follows:
$$\begin{tikzpicture}[scale=0.7]
\draw [-] (-9,0) -- (8.5,0);
\draw (-9,0.1) node[above]{$r=$} ;
\draw (-9,-0.1) node[below]{$\theta_r=$} ;

\draw (-7.8,0.1) node[above]{$-k$} ;
\draw (-7.8,0) node  {$\bullet$} ;
\draw (-7.8,-0.1) node[below]{$\frac{1}{2^{\frac{k-2}{2}}}$} ;

\draw (-6.8,0.2) node[above]{$\cdots$} ;
\draw (-6.8,-0.3) node[below]{$\ldots$} ;

\draw (-6,0.1) node[above]{$-6$} ;
\draw (-6,0) node  {$\bullet$} ;
\draw (-6,-0.1) node[below]{$\frac{1}{4}$} ;

\draw (-5,0.1) node[above]{$-5$} ;
\draw (-5,0) node  {$\bullet$} ;
\draw (-5,-0.1) node[below]{$0$} ;

\draw (-4,0.1) node[above]{$-4$} ;
\draw (-4,0) node  {$\bullet$} ;
\draw (-4,-0.1) node[below]{$\frac{1}{2}$} ;

\draw (-3,0.1) node[above]{$-3$} ;
\draw (-3,0) node  {$\bullet$} ;
\draw (-3,-0.1) node[below]{$0$} ;

\draw (-2,0.1) node[above]{$-2$} ;
\draw (-2,0) node  {$\bullet$} ;
\draw (-2,-0.1) node[below]{$1$} ;

\draw (-1,0.1) node[above]{$-1$} ;
\draw (-1,0) node  {$\bullet$} ;
\draw (-1,-0.1) node[below]{$1$} ;

\draw (0,0.1) node[above]{$0$} ;
\draw (0,0) node  {$\bullet$} ;
\draw (0,-0.1) node[below]{$-1$} ;

\draw (1,0.1) node[above]{$1$} ;
\draw (1,0) node  {$\bullet$} ;
\draw (1,-0.1) node[below]{$-1$} ;

\draw (2,0.1) node[above]{$2$} ;
\draw (2,0) node  {$\bullet$} ;
\draw (2,-0.1) node[below]{$-2$} ;

\draw (3,0.1) node[above]{$3$} ;
\draw (3,0) node  {$\bullet$} ;
\draw (3,-0.1) node[below]{$\frac{1}{2}$} ;

\draw (4,0.1) node[above]{$4$} ;
\draw (4,0) node  {$\bullet$} ;
\draw (4,-0.1) node[below]{$0$} ;

\draw (5,0.1) node[above]{$5$} ;
\draw (5,0) node  {$\bullet$} ;
\draw (5,-0.1) node[below]{$\frac{1}{4}$} ;

\draw (6,0.1) node[above]{$6$} ;
\draw (6,0) node  {$\bullet$} ;
\draw (6,-0.1) node[below]{$0$} ;

\draw (6.8,0.2) node[above]{$\cdots$} ;
\draw (6.8,-0.3) node[below]{$\ldots$} ;

\draw (7.8,0.1) node[above]{$k$} ;
\draw (7.8,0) node  {$\bullet$} ;
\draw (7.8,-0.1) node[below]{$\frac{1}{2^{\frac{k-1}{2}}}$} ;
\end{tikzpicture}$$
Then, again, $\theta$ satisfies the conditions of Definition \ref{deftheta}. We now have $D_-=\{0,1,2\}$, 
hence there are two $(\OO_X,G)$-submodules of $\FF$ generated in $D_-$ which are $\FF_1:=I_2$ and $\FF_2:=I_1$, 
with $r(\FF_{1})=1$ and $r(\FF_{2})=2$. They verify 
$$ \mut(\FF_1)=\mut(\FF_2)=1>\mut(\FF)=0,$$
and thus the $\mut$-HN filtration of $\FF$ is 
$$0 \subsetneq \FF_2 \subsetneq \FF.$$

Now, let us distinguish between two situations: the \emph{even} and the \emph{odd} cases.
For the even case we take $D=D_{N,\even}:=[-2N-2,2N+2]$, and for the odd case we take $D=D_{N,\odd}:=[-2N-3,2N+3]$, where these bounds are chosen in order to simplify the calculations. In both cases, there exists $N_0 \geq 1$ such that $D_{N,\bullet}$ satisfies Hypothesis \ref{hypoD} for all $N \geq N_0$.   

In the even case, we have
\begin{align*}
\mu_{D_{N, \even}}(\FF_{1})&=\frac{-1}{r(\FF_{1})} \left( \sum_{r\in D_{N,\even}} \theta_r h_{I_{2}}(r)+ \frac{S_{D_{N, \even}}}{d} 
\sum_{r\in D_{N, \even} \setminus D_-} \frac{h_{I_{2}}(r)}{h(r)} \right)\\
&=(-1)\cdot \left((-1-\frac{1}{2^{N}})+\frac{\frac{1}{2^{N-1}}}{4N+2}(2N)\right)=1+\frac{1}{2^{N}}\left(\frac{2}{4N+2}\right)\; ; \text{ and }\\
\mu_{D_{N, \even}}(\FF_{2})&=\frac{-1}{r(\FF_{2})} \left( \sum_{r\in D_{N,\even}} \theta_r h_{I_{1}}(r)+ \frac{S_{D_{N, \even}}}{d} 
\sum_{r\in D_{N, \even} \setminus D_-} \frac{h_{I_{1}}(r)}{h(r)} \right)\\
&=\left(\frac{-1}{2}\right)\cdot \left((-2-\frac{1}{2^{N}})+\frac{\frac{1}{2^{N-1}}}{4N+2}(2N)\right)=1+\frac{1}{2^{N+1}}\left(\frac{2}{4N+2}\right)\; .
\end{align*}
Observe that $\mu_{D_{N,\even}}(\FF_{1})>\mu_{D_{N,\even}}(\FF_{2})>0$, hence the $\mu_{D_{N,\even}}$-HN filtration of $\FF$ is
$$0\subsetneq \FF_{1}\subsetneq \FF_{2}\subsetneq \FF\; .$$

Performing analogous calculations in the odd case, we get
\begin{align*}
\mu_{D_{N,\odd}}(\FF_{1})&=1+\frac{1}{2^{N+1}}\left(\frac{-2N+1}{4N+4}\right)\; ; \text{ and}\\
\mu_{D_{N,\odd}}(\FF_{2})&=1+\frac{1}{2^{N+2}}\left(\frac{-2N+1}{4N+4}\right)\; .
\end{align*}
We see that $\mu_{D_{N,\odd}}(\FF_{2})>\mu_{D_{N,\odd}}(\FF_{1})>0$, hence the $\mu_{D_{N,\odd}}$-HN filtration of $\FF$ is 
$$0\subsetneq \FF_{2}\subsetneq \FF\; .$$

Therefore, in the \emph{odd} case, the $\mu_{D_{N,\odd}}$-HN filtration has exactly the same terms as the $\mut$-HN filtration, but in the 
\emph{even} case, the $\mu_{D_{N,\even}}$-HN filtration has one more term. 
Figure \ref{pictureexevenodd} illustrates this behavior by showing the $\theta$-polygon and the $D_{N}$-polygons of $\FF$.
Also, observe how the $D_{N,\even}$-polygons lie above the $\theta$-polygon, the $D_{N,\odd}$-polygons lie below, and both sequences of polygons converge to the $\theta$-polygon when $N$ grows (as stated by Theorem \ref{convergence}).

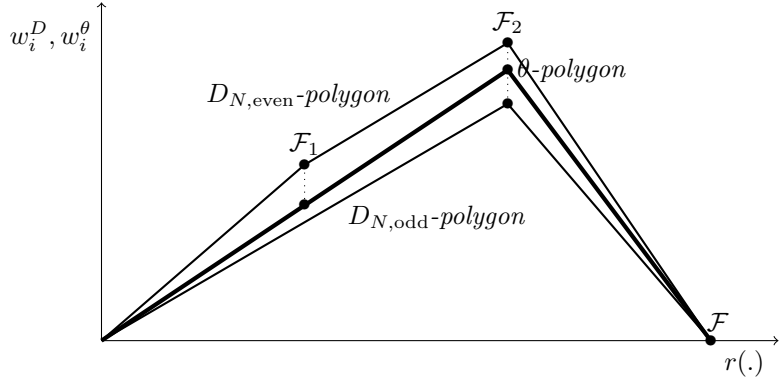
\begin{figure}[h]
\begin{center}
$$\begin{tikzpicture}[scale=0.9]
\draw [->] (-5,0) -- (5,0);
\draw [->] (-5,0) -- (-5,5);
\node [left] at  (-5,4.5) {$w^{D}_{i}, w^{\theta}_{i}$};
\node [below] at (4.5,0) {$r(.)$};
\draw [-,ultra thick] (-5,0) -- (-2,2);
\draw (-2,2) node  {$\bullet$} ;
\draw [-,ultra thick] (-2,2) -- (1,4);
\draw (1,4) node  {$\bullet$} ;
\draw [-, ultra thick] (1,4) -- (4,0);
\draw (4,0) node  {$\bullet$} ;
\draw [-, thick] (-5,0) -- (-2,2.6);
\draw (-2,2.6) node  {$\bullet$} ;
\draw [-, thick] (-2,2.6) -- (1,4.4);
\draw (1,4.4) node  {$\bullet$} ;
\draw [-, thick] (1,4.4) -- (4,0);
\draw [-, thick] (-5,0) -- (1,3.5);
\draw (1,3.5) node  {$\bullet$} ;
\draw [-, thick] (1,3.5) -- (4,0);
\draw [dotted] (1,3.5) -- (1,4.4);
\draw [dotted] (-2,2) -- (-2,2.6);
\node [above] at (1,4.4) {$\FF_{2}$};
\node [above] at (-2,2.6) {$\FF_{1}$};
\node [above] at (4.1,0) {$\FF$};

\node [right] at (1,4) {$\theta$\text{-polygon}};
\node [left] at (-0.6,3.6) {$D_{N,\even}$\text{-polygon}};
\node [right] at (-1.5,1.8) {$D_{N,\odd}$\text{-polygon}};
\end{tikzpicture}$$
\caption{$\theta$-polygon and $D_{N}$-polygons of Example \ref{ex3}}
\label{pictureexevenodd}
\end{center}
\end{figure}
\end{example}

To compute our final example we will slightly change the framework and pick another Hilbert function $h$. 
Consider $G=\Gm$ with the same action on the algebra $\CC[x,y]$ as before, but take now $X:=\Spec \CC[x,y]/(xy^{2},x^{3}y)$. 
Then we have
\begin{align*}
\CC[x,y]/(xy^{2},x^{3}y) &\cong \CC[x]_{>0} \;\oplus \;\CC \;\oplus \;\CC<\overline{xy}>\; \oplus \;\CC<\overline{x^2y}>\; \oplus \; \CC[y]_{>0}\\
                         &\cong \left(\bigoplus_{r \in \ZZ \setminus \{0,1\}} V_r \right) \oplus V_0^{\oplus 2} \oplus V_1^{\oplus 2}.
\end{align*}                         
as $G$-modules, where $\overline{x}$ and $\overline{y}$ denote the images of $x$ and $y$ in $\CC[x,y]/(xy^{2},x^{3}y)$.  
Let $h: \ZZ \to \NN$ be the Hilbert function defined by $h(r)=2$ if $r=0,1$, and $h(r)=1$ for all $r\neq 0,1$.
It is clear that $\OO_X$ is a $(G,h)$-constellation on $X$, provided that we choose $\theta$ such that $\theta_{0}<0$ (to guarantee that $\OO_X$ is generated in $D_-$).

\begin{example} \label{ex4}
Example where, for all $D$ big enough, the $\mut$-HN and $\mu_D$-HN filtrations are both non-trivial, and the $\mut$-HN filtration 
is a strict subfiltration of the $\mu_D$-HN filtration which is, in turn, a strict subfiltration of some $\mut$-Jordan-H\"older filtration 
(see Remark \ref{HNtoJH}).  
Let $\theta$ be defined as follows:
$$\begin{tikzpicture}
\draw [-] (-5.7,0) -- (6.8,0);
\draw (-5.2,0) node[above]{$r=$} ;
\draw (-5.2,0) node[below]{$\theta_{r}=$} ;

\draw (-4.4,0) node[above]{$-k$} ;
\draw (-4.4,0) node  {$\bullet$} ;
\draw (-4.4,0) node[below]{$0$} ;

\draw (-3.7,0) node[above]{$\cdots$} ;
\draw (-3.7,0) node[below]{$\ldots$} ;

\draw (-3,0) node[above]{$-3$} ;
\draw (-3,0) node  {$\bullet$} ;
\draw (-3,0) node[below]{$0$} ;

\draw (-2,0) node[above]{$-2$} ;
\draw (-2,0) node  {$\bullet$} ;
\draw (-2,0) node[below]{$0$} ;

\draw (-1,0) node[above]{$-1$} ;
\draw (-1,0) node  {$\bullet$} ;
\draw (-1,0) node[below]{$5$} ;

\draw (0,0) node[above]{$0$} ;
\draw (0,0) node  {$\bullet$} ;
\draw (0,0) node[below]{$-1$} ;

\draw (1,0) node[above]{$1$} ;
\draw (1,0) node  {$\bullet$} ;
\draw (1,0) node[below]{$-1$} ;

\draw (2,0) node[above]{$2$} ;
\draw (2,0) node  {$\bullet$} ;
\draw (2,0) node[below]{$-2$} ;

\draw (3,0) node[above]{$3$} ;
\draw (3,0) node  {$\bullet$} ;
\draw (3,0) node[below]{$\frac{1}{2}$} ;

\draw (4,0) node[above]{$4$} ;
\draw (4,0) node  {$\bullet$} ;
\draw (4,0) node[below]{$\frac{1}{4}$} ;

\draw (5,0) node[above]{$5$} ;
\draw (5,0) node  {$\bullet$} ;
\draw (5,0) node[below]{$\frac{1}{8}$} ;

\draw (5.7,0) node[above]{$\cdots$} ;
\draw (5.7,0) node[below]{$\ldots$} ;

\draw (6.4,0) node[above]{$k$} ;
\draw (6.4,0) node  {$\bullet$} ;
\draw (6.4,0) node[below]{$\frac{1}{2^{k-2}}$} ;
\end{tikzpicture}$$
Let $\FF:=\OO_X$ and observe that $\theta(\FF)=0$, so that $\theta$ satisfies the conditions of Definition \ref{deftheta}. 
There are five non-zero proper $(\OO_X,G)$-submodules of $\FF$ generated in $D_-=\{0,1,2\}$, say
$$\FF_{1}:=(\overline{x}^{2}\overline{y}),\  \FF_{2}:=(\overline{x}^{2}),\ \FF_{3}:=(\overline{x}^{2}, \overline{x}\overline{y}),\ 
\FF_{4}:=(\overline{x}), \text{\ and\ }\,\FF_{5}:=(\overline{x}\overline{y}).$$ 
These submodules verify
$$r(\FF_{1})=1,\ r(\FF_{2})=2,\ r(\FF_{3})=3,\ r(\FF_{4})=4, \text{\ and\ }\,r(\FF_{5})=2.$$ 
For the sake of completeness, we detail their Hilbert functions which are
\begin{align*}
h_{\FF_1}(r)=\left\{
\begin{array}{l}
  1,\; r=1\\
  0,\; r\neq 1
\end{array}
\right\};\;\;
& h_{\FF_2}(r)=\left\{
\begin{array}{l}
  1,\; r\geq 1\\
  0,\; r\leq 0
\end{array}
\right\};
& h_{\FF_3}(r)=\left\{
\begin{array}{l}
  1,\; r\geq 0\\
  0,\; r\leq -1
\end{array}
\right\}; 
\end{align*}
\begin{align*}
h_{\FF_4}(r)=\left\{
\begin{array}{l}
  1,\; r=0\\
  2,\; r=1\\
  1,\; r\geq 2\\
  0,\; r\leq -1
\end{array}
\right\};\;\; 
& h_{\FF_5}(r)=\left\{
\begin{array}{l}
  1,\; r=0,1\\
  0,\; r\neq 0,1
\end{array}
\right\}.
\end{align*}

First, we see that
$$\mut(\FF_{1})=\mut(\FF_{2})=\mut(\FF_{3})=\mut(\FF_{4})=\mut(\FF_{5})=1>0=\mut(\FF)\; ,$$
hence the $\mut$-HN filtration of $\FF$ is
$$0\subsetneq \FF_{4}\subsetneq \FF\; .$$

Now, let $D=D_{N}:=[-N,N] \subset \ZZ$. There exists $N_0 \geq 1$ such that $D_N$ satisfies Hypothesis \ref{hypoD} for all $N \geq N_0$. 
Let us determine the $\mu_{D_N}$-HN filtration of $\FF$. 
Analogous calculations to those performed in the previous examples give
$$\mu_{D_{N}}(\FF_{1})=1\; ;$$
$$\mu_{D_{N}}(\FF_{2})=1+\frac{1}{2^{N-1}}\left(\frac{N-\frac{3}{2}}{2N-2}\right)\; ;$$
$$\mu_{D_{N}}(\FF_{3})=1+\frac{1}{3\cdot 2^{N-2}}\left(\frac{N-2}{2N-2}\right)\; ;$$
$$\mu_{D_{N}}(\FF_{4})=1+\frac{1}{2^{N}}\left(\frac{N-\frac{5}{2}}{2N-2}\right)\; ;$$
$$\mu_{D_{N}}(\FF_{5})=1\; .$$
Also, asymptotically we have 
$$\mu_{D_{N}}(\FF_{2})\sim 1+\frac{1}{2^{N}}\; ;$$
$$\mu_{D_{N}}(\FF_{3})\sim 1+\frac{1}{3\cdot 2^{N-1}}\; ;$$
$$\mu_{D_{N}}(\FF_{4})\sim 1+\frac{1}{2^{N+1}}\; .$$
Hence, there exists $N_{1}\geq N_{0}$ such that, for all $N\geq N_{1}$, the $\mu_{D_N}$-HN filtration of $\FF$ is  
$$0\subsetneq \FF_{2}\subsetneq \FF_{3}\subsetneq \FF_{4}\subsetneq\FF\; .$$

Observe that, for $N\geq N_{1}$, the $\mu_{D_N}$-HN filtration contains the unique non trivial term of the $\mut$-HN filtration, $\FF_{4}$, (as proved
in Theorem \ref{HNterms}) but it contains also two more terms, $\FF_{2}$ and $\FF_{3}$. On the other hand, the first 
$\mut$-semistable factor of the $\mut$-HN filtration of $\FF$, which is $\FF_{4}$, has two different $\mut$-Jordan-H\"older filtrations
$$0\subsetneq \FF_{1}\subsetneq \FF_{2}\subsetneq \FF_{3}\subsetneq \FF_{4}\;, \text{   and     }
\;0\subsetneq \FF_{1}\subsetneq \FF_{5}\subsetneq \FF_{3}\subsetneq \FF_{4}\; ,$$
and the $\mu_{D_N}$-HN filtration is a subfiltration of the first one but not of the second one (see Remark \ref{HNtoJH}). 
Figure \ref{pictureexample4} illustrates the situation.  
\begin{figure}[h]
\begin{center}
$$\begin{tikzpicture}[scale=0.9]
\draw [->] (-5,0) -- (6,0);
\draw [->] (-5,0) -- (-5,5);
\node [left] at  (-5,4.5) {$w^{D}_{i}, w^{\theta}_{i}$};
\node [below] at (5.5,0) {$r(.)$};
\draw [-,ultra thick] (-5,0) -- (3,4);
\draw (-3,1) node  {$\bullet$} ;
\draw (-1,2) node  {$\bullet$} ;
\draw (1,3) node  {$\bullet$} ;
\draw (3,4) node  {$\bullet$} ;
\draw [-, ultra thick] (3,4) -- (5,0);
\draw (5,0) node  {$\bullet$} ;
\draw [-, thick] (-5,0) -- (-1,3);
\draw (-1,3) node  {$\bullet$} ;
\draw [-, thick] (-1,3) -- (1,4);
\draw (1,4) node  {$\bullet$} ;
\draw [-, thick] (1,4) -- (3,4.6);
\draw (3,4.6) node  {$\bullet$} ;
\draw [-, thick] (3,4.6) -- (5,0);
\draw [dotted] (-1,2) -- (-1,3);
\draw [dotted] (1,3) -- (1,4);
\draw [dotted] (3,4) -- (3,4.6);
\node[below] at (-3,0.9) {$\FF_{1}$};
\node [below] at (-1,1.9) {$\FF_{2}$};
\node [below] at (1,2.8) {$\FF_{3}$};
\node [below] at (2.9,3.8) {$\FF_{4}$};
\node [above] at (-1,3.1) {$\FF_{2}$};
\node [above] at (1,4.1) {$\FF_{3}$};
\node [above] at (3,4.6) {$\FF_{4}$};
\node [above] at (5.2,0) {$\FF$};

\node [below right] at (-0.2,2.4) {$\theta$\text{-polygon}};
\node [above left] at (0.6,3.7) {$D_{N}$\text{-polygon}};
\end{tikzpicture}$$
\caption{$\theta$-polygon and $D_{N}$-polygons of Example \ref{ex4}}
\label{pictureexample4}
\end{center}
\end{figure}
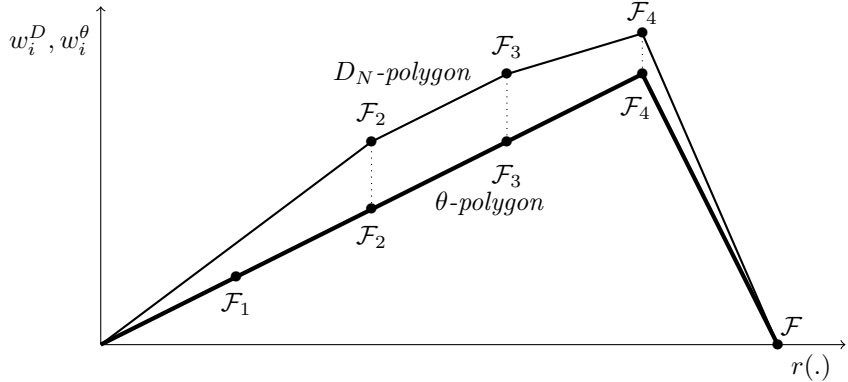
\end{example}

\begin{remark}
For the reader wiling to consider fancier situations, other examples of $(G,h)$-constellations with $G$ a classical group can be found in \cite{Ter12}.
\end{remark}

\noindent \textbf{Acknowledgements.}
We thank the referee for his/her review and highly appreciate the comments and suggestions he/she made which significantly contributed to improving the quality of our article. 
 The first-named author is grateful to Christian Lehn for suggesting the correct definition of $\theta$-stability in the setting 
of $(G,h)$-constellations and for helpful discussions. The second-named author thanks the Johannes Gutenberg Universit\"at in Mainz for the 
hospitality it provided while part of this work was done. The first-named author benefits from the support of the DFG via the SFB/TR 45 
''Periods, Moduli Spaces and Arithmetic of Algebraic Varieties''. 
The second-named author is supported by the project \textquotedblleft Comunidade Portuguesa de Geometr\'ia Algebrica\textquotedblright
\;PTDC/MAT-GEO/0675/2012 funded by Portuguese FCT and project MTM2016-79400-P granted by the Spanish Ministerio de Econom\'ia y Competitividad.

\end{document}